\documentclass[11pt,a4paper]{amsart}
\usepackage{amssymb,amsmath}
\usepackage{mathrsfs}
\usepackage[alphabetic]{amsrefs}
\usepackage{graphicx,color}
\usepackage[utf8]{inputenc}
\usepackage{amsthm}
\usepackage{latexsym}
\usepackage{slashed}
\usepackage[all]{xy}
\usepackage{hyperref}
\usepackage[mathscr]{eucal}
\usepackage{tikz}
\usepackage{mathtools}
\usepackage{accents}

\def\theequation{\thesection.\arabic{equation}}
\makeatletter
\@addtoreset{equation}{section}
\makeatother

\makeatletter
\newcommand{\eqnum}{\refstepcounter{equation}\textup{\tagform@{\theequation}}}
\makeatother

\newcounter{copy}
\makeatletter
\renewcommand{\thecopy}{\ifnum0=\c@section\arabic{copy}\else\thesection.\arabic{copy}'\fi}
\makeatother
\newcommand{\copynum}[1][]{\refstepcounter{copy}#1{\thecopy}} 

\BibSpec{book}{%
    +{}  {\PrintPrimary}                {transition}
    +{.} { \textit}                     {title}
    +{.} { }                            {part}
    +{:} { \textit}                     {subtitle}
    +{,} { \PrintEdition}               {edition}
    +{}  { \PrintEditorsB}              {editor}
    +{,} { \PrintTranslatorsC}          {translator}
    +{,} { \PrintContributions}         {contribution}
    +{,} { }                            {series}
    +{,} { \voltext}                    {volume}
    +{,} { }                            {publisher}
    +{,} { }                            {organization}
    +{,} { }                            {address}
    +{,} { }                            {status}
    +{,} { \PrintDOI}                   {doi}
    +{,} { \PrintISBNs}                 {isbn}
    +{}  { \parenthesize}               {language}
    +{}  { \PrintTranslation}           {translation}
    +{;} { \PrintReprint}               {reprint}
    +{,} { \PrintDate}                  {date}
    +{.} { }                            {note}
    +{.} {}                             {transition}
}
\BibSpec{article}{%
    +{}  {\PrintAuthors}                {author}
    +{,} { \textit}                     {title}
    +{.} { }                            {part}
    +{:} { \textit}                     {subtitle}
    +{,} { \PrintContributions}         {contribution}
    +{.} { \PrintPartials}              {partial}
    +{,} { }                            {journal}
    +{}  { \textbf}                     {volume}
    +{}  { \PrintDatePV}                {date}
    +{,} { \issuetext}                  {number}
    +{,} { \eprintpages}                {pages}
    +{,} { }                            {status}
    +{,} { \PrintDOI}                   {doi}
    +{}  { \parenthesize}               {language}
    +{}  { \PrintTranslation}           {translation}
    +{;} { \PrintReprint}               {reprint}
    +{.} { }                            {note}
    +{,} { \eprint}                     {eprint}
    +{.} {}                             {transition}
}
\BibSpec{collection.article}{%
    +{}  {\PrintAuthors}                {author}
    +{,} { \textit}                     {title}
    +{.} { }                            {part}
    +{:} { \textit}                     {subtitle}
    +{,} { \PrintContributions}         {contribution}
    +{,} { \PrintConference}            {conference}
    +{}  {\PrintBook}                   {book}
    +{,} { }                            {booktitle}
    +{,} { \PrintDateB}                 {date}
    +{,} { pp.~}                        {pages}
    +{,} { }                            {publisher}
    +{,} { }                            {organization}
    +{,} { }                            {address}
    +{,} { }                            {status}
    +{,} { \PrintDOI}                   {doi}
    +{,} { \eprint}        {eprint}
    +{}  { \parenthesize}               {language}
    +{}  { \PrintTranslation}           {translation}
    +{;} { \PrintReprint}               {reprint}
    +{.} { }                            {note}
    +{.} {}                             {transition}
}

\BibSpec{misc}{
  +{}{\PrintAuthors}  {author}
  +{,}{ \textit}     {title}
  +{}{ (}             {date}
  +{),}{ }             {note}
  +{.}{}              {transition}
}

\theoremstyle{definition}
\newtheorem{defn}[equation]{Definition}
\newtheorem{notn}[equation]{Notation}
\theoremstyle{plain}
\newtheorem{thm}[equation]{Theorem}
\newtheorem{prp}[equation]{Proposition}
\newtheorem{lem}[equation]{Lemma}
\newtheorem{cor}[equation]{Corollary}

\theoremstyle{remark}
\newtheorem{rmk}[equation]{Remark}

\newcommand{\bB}{\mathbb{B}}
\newcommand{\bC}{\mathbb{C}}
\newcommand{\bD}{\mathbb{D}}

\newcommand{\bK}{\mathbb{K}}

\newcommand{\bM}{\mathbb{M}}

\newcommand{\bQ}{\mathbb{Q}}
\newcommand{\bR}{\mathbb{R}}

\newcommand{\bT}{\mathbb{T}}

\newcommand{\bV}{\mathbb{V}}

\newcommand{\bZ}{\mathbb{Z}}

\newcommand{\cE}{\mathcal{E}}

\newcommand{\cM}{\mathcal{M}}

\newcommand{\cP}{\mathcal{P}}
\newcommand{\cQ}{\mathcal{Q}}

\newcommand{\cU}{\mathcal{U}}
\newcommand{\cV}{\mathcal{V}}
\newcommand{\cW}{\mathcal{W}}
\newcommand{\cX}{\mathcal{X}}

\newcommand{\sE}{\mathscr{E}}

\newcommand{\geom}{\mathit{geom}}
\newcommand{\APS}{\mathrm{APS}}
\newcommand{\CWY}{\mathrm{CWY}}
\newcommand{\DG}{\mathrm{DG}}

\newcommand{\MF}{\mathrm{MF}}

\renewcommand{\Im}{\mathrm{Im} \hspace{0.1em}}
\newcommand{\pt}{\mathrm{pt}}

\newcommand{\id}{\mathrm{id}}
\newcommand{\ev}{\mathrm{ev}}
\newcommand{\Cl}{\text{\it C$\ell$}}
\newcommand{\Cay}{\mathrm{Cay}}

\DeclareMathOperator*{\limone}{\varprojlim\nolimits^1}
\DeclareMathOperator{\hotimes}{\hat{\otimes }}
\DeclareMathOperator{\ch}{\mathrm{ch}} 
\DeclareMathOperator{\K}{\mathrm{K}}
\DeclareMathOperator{\KR}{\mathrm{KR}}
\DeclareMathOperator{\KO}{\mathrm{KO}}
\DeclareMathOperator{\KK}{\mathrm{KK}}

\DeclareMathOperator{\KKR}{\mathrm{KKR}}

\DeclareMathOperator{\Hom}{\mathrm{Hom}}
\DeclareMathOperator{\ind}{\mathrm{ind}}
\DeclareMathOperator{\supp}{\mathrm{supp}}
\DeclareMathOperator{\Ad}{\mathrm{Ad}}

\DeclareMathOperator{\Res}{Res}
\newcommand{\Csep}{\mathfrak{C}^*\mathfrak{sep}}
\newcommand{\Kas}{\mathfrak{KK}}

\title[The relative higher index and almost flat bundles I]{The relative Mishchenko--Fomenko higher index and almost flat bundles I:\\ The relative Mishchenko--Fomenko index}
\author{Yosuke KUBOTA}
\address{iTHEMS Research Program, RIKEN, 2-1 Hirosawa, Wako, Saitama 351-0198, Japan}
\email{yosuke.kubota@riken.jp}

\date{}
\subjclass[2010]{Primary 19K56; Secondary 19K35, 46L80, 58J32.}
\keywords{Chang--Weinberger--Yu relative higher index, positive scalar curvature metric, almost flat bundle, $\KK$-theory.}

\begin{document}
\begin{abstract}
In this paper, the first of two, we introduce an alternative definition of the Chang--Weinberger--Yu relative higher index, which is thought of as a relative analogue of the Mishchenko--Fomenko index pairing. A main result of this paper is that our map coincides with the existing relative higher index maps.
We make use of this fact for understanding the relative higher index.  First, we relate the relative higher index with the higher index of amalgamated free product groups. Second, we define the dual relative higher index map and show its rational surjectivity under certain assumptions.
\end{abstract}

\maketitle
\tableofcontents

\section{Introduction}
This is the first of a series of two papers which aim to reveal a relation between the \emph{Chang--Weinberger--Yu relative higher index} and the Riemannian geometry of almost flat vector bundles on manifolds with boundary. 
The relative higher index is a new index theoretic invariant for spin manifolds with boundary introduced in \cite{mathKT150603859}. It is an obstruction to a collared Riemannian metric with positive scalar curvature. 
The aim of this paper is to give a new equivalent definition of this invariant by using Kasparov's KK-theory. In the sequel \cite{Kubota2} of this paper, we make use of our definition to relate the relative higher index with the geometry of hermitian vector bundles on Riemannian manifolds.

For a compact space $X$ with a reference map to the classifying space $B\Gamma$ of a countable discrete group $\Gamma$, the higher index map for $X$ is a homomorphism
\[ \mu^\Gamma_* \colon \K_*(X) \to \K_*(C^*(\Gamma)), \]
where $C^*\Gamma$ is the maximal group C*-algebra. When $\Gamma$ is torsion-free, the the higher index map for $B\Gamma $ is called the the Baum--Connes assembly map and conjectured to be an isomorphism. 
The higher index map is defined in several ways with different backgrounds, each of which has an advantage.
Kasparov's definition \cite{MR918241} using the descent functor in equivariant KK-theory enables us to define the assembly map with coefficients.
In coarse index theory \cite{MR1399087}, the higher index is realized as the boundary homomorphism of an exact sequence of C*-algebras and hence described without any use of $\KK$-theory.
The Mishchenko--Fomenko higher index \cite{MR548506}, a main subject of this paper, is defined as the Kasparov product with the element of $ \K_0(C(X) \otimes C^*(\Gamma ))$ represented by the Mishchenko line bundle, and hence we can define the dual assembly map
\[ \beta_{\Gamma} \colon \K^*(C^*\Gamma) \to \K^*(B\Gamma) \]
which is dual to $\mu_*^\Gamma$.

Let $\Lambda$ be another countable discrete group with a homomorphism $\phi \colon \Lambda \to \Gamma$. 
For a pair $(X,Y)$ of compact spaces with a reference map to $(B\Gamma, B\Lambda)$, the relative higher index is a homomorphism
\[\mu_*^{\Gamma,\Lambda} \colon \K(X,Y) \to \K_*(C^*(\Gamma , \Lambda)). \]
Here, the relative group C*-algebra $C^*(\Gamma, \Lambda)$ is defined as the suspension of the mapping cone of the homomorphism between maximal group C*-algebras. It was first defined by \cite{mathKT150603859} as a relative version of the coarse higher index. After that, Deeley and Goffeng gave an alternative definition in \cite{mathKT150701735,mathKT170508467} from the viewpoint of Baum--Douglas geometric K-homology and its variations developed in \cite{MR3519049}. 
In Section \ref{section:3}, we provide the third definition of the relative higher index; a relative version of the Mishchenko--Fomenko index pairing. 
Here we construct the `relative Mishchenko line bundle' as an element of the group $\K_0(C_0(X^\circ ) \otimes C^*(\Gamma, \Lambda))$, where $X^\circ$ is the interior $X \setminus Y$, and define the map as the Kasparov product with it. 
A main theorem of this paper, studied in Section \ref{section:3.5}, is that our definition actually coincides with the Chang--Weinberger--Yu and Deeley--Goffeng maps.

For the rest two sections, we display applications of our relative higher index map. First, in Section \ref{section:9} we clarify the relation between the relative higher index and the higher index of amalgamated free product groups. That is, for a manifold $M$ partitioned by a hypersurface $N$ as $M=M_1 \sqcup _N M_2$, we relate the relative higher indices of $M_1$ and $M_2$ with the higher index of $M$. A key observation is the $\KK$-equivalence of the group C*-algebra of the amalgamated free product with a mapping cone C*-algebra, which is essentially proved by Pimsner~\cite{MR860685}. This relation is useful to show the non-vanishing of from one to another. In particular, we consider two geometric situations. First, we get the non-vanishing of the higher index of $M$ from that of the hypersurface $N$ in a different approach to \cite{MR2670972}. Second, we discuss the invariance of the non-vanishing of the higher index under the cutting-and-pasting of a manifold along a hypersurface. 

Second, we study the dual relative higher index map. In the same way as the ordinal Baum--Connes assembly map, our definition enables us consider the dual
\[\beta_{\Gamma ,\Lambda} \colon \K^*(C^*(\Gamma, \Lambda)) \to \K^*(X,Y ) \]  
of the relative higher index map (Definition \ref{defn:dBC}).   In Section \ref{section:4}, we investigate the rational surjectivity of this map for the pair $(B\Gamma, B\Lambda )$, which is closely related to the rational injectivity of the relative assembly map.  The key ingredient of the proof is the Dirac--dual Dirac method~\cite{MR918241}. 
The results in this section will be used in the second part \cite{Kubota2} to study the subgroup of almost flat elements in the relative $\K^0$-group $\K^0(X,Y)$ under certain assumptions of fundamental groups.

\begin{notn}\label{notation}
Throughout this paper we use the following notations.
\begin{itemize}
\item For a C*-algebra $A$, let $A^+$ denote its unitization $A + \bC \cdot 1$. 
\item For a C*-algebra $A$, let $\cM(A)$ denote its multiplier C*-algebra and $\cQ(A):=\cM(A)/A$.
\item For a C*-algebra $A$ and $a<b \in \bR \cup \{ \pm \infty \}$,let $A(a,b):=A \otimes C_0(a,b)$. Similarly we define $A[a,b)$ and $A[a,b]$. For a Hilbert $A$-module $E$, let $E(a,b)$ denote the Hilbert $A(a,b)$-module $E \otimes C_0(a,b)$.
\item For a $\ast$-homomorphism $\phi \colon A \to B$, let $C\phi$ denote the mapping cone C*-algebra defined as
\[ C\phi = \{(a, b_s) \in A \oplus B[0,1) \mid \phi(a)=b_0 \}.\] 
\item For a Hilbert $A$-module $E$, let $\bB (E)$ and $\bK(E)$ denote the C*-algebra of bounded adjointable and compact operators on $E$ respectively.
\item We say that a Real C*-algebra is a C*-algebra equipped with an antilinear involutive $\ast$-isomorphism $a \mapsto \bar{a}$. An element $a \in A$ is said to be real if $\bar{a}=a$. In particular, let $S^{0,1}$ denote the Real C*-algebra $C_0(\bR)$ with the complex conjugation $\overline{f}(x):=\overline{f(-x)}$.
\item Let $A$ be an ungraded C*-algebra and let $E=E^0 \oplus E^1$ be a $\bZ_2$-graded Hilbert $A$-module. For an odd operator $F \in \bB(E)$, we write $F^0 \in \bB(E^0,E^1)$ and $F^1 \in \bB(E^1,E^0)$ for the operators such that 
\[
F=\begin{pmatrix} 0 & F^1 \\ F^0 & 0 \end{pmatrix}_{\textstyle .} 
\]
Note that $F^1=(F^0)^*$ if $F$ is self-adjoint.
\end{itemize}
\end{notn}

\subsection*{Acknowledgment}
The author would like to thank John Roe for his helpful discussion and encouragement. He also thank the anonymous referee for his careful reading of the paper and helpful advices. This work was supported by RIKEN iTHEMS Program.

\section{Preliminary on relative higher index maps}\label{section:2}
In this section, we review the existing two definitions of the relative higher index for self-consistency of the paper.  

\begin{notn}\label{not:Xr}
For a pair $(X,Y)$ of locally compact spaces with a deformation retract neighborhood $U$ of $Y$ (for example, a pair of finite CW-complexes or finite $\Gamma$-CW-complexes) and $r \in [0,\infty]$, set 
\begin{align*}
Y_r&:= \left\{ \begin{array}{ll} Y \times [0,r] & \text{for $r \in [0,\infty) $}, \\ Y \times [0,\infty) & \text{for $r = \infty$},  \end{array} \right. \\
X_r &:= X \sqcup _Y Y_r.
\end{align*}
For $r \in [1,\infty)$, let $Y_r':= Y \times [1, r] \subset X_\infty$. 
\end{notn}

Let $\Gamma$ and $\Lambda$ be countable discrete groups and let $\phi \colon \Lambda \to \Gamma$ be a group homomorphism. Then we have a continuous map $B\phi \colon B\Lambda \to B\Gamma $. Since both $B\Gamma$ and $B\Lambda$ have the homotopy type of CW-complexes, we may replace $B\Gamma$ with the mapping cylinder $B\Gamma \sqcup_{B\phi} B\Lambda [0,1]$ and $B\Lambda $ with $B\Lambda \times \{ 1 \}$ to assume that $B\phi$ is injective.
 
The domain of the relative higher index map is the K-homology or $\KO$-homology group of the pair $(X,Y)$ of finite CW-complexes with a reference map $(X,Y) \to (B\Gamma , B\Lambda )$, to which the pull-back Galois coverings $\Gamma \to \tilde{X} \to X$ and $\Lambda \to \tilde{Y} \to Y$ are associated. In this paper we employ the Baum--Douglas geometric K-homology among the equivalent definitions of the K-homology group. That is, an element of $\KO_i(X, Y)$ is represented by a triplet $(M,f,E)$, where 
\begin{itemize}
\item $M$ is a compact spin manifold of dimension $8n+i$ with the boundary $N$, 
\item $f \colon (M, N ) \to (X,Y)$ is a continuous map of pairs, and 
\item $E$ is a real vector bundle on $M$.
\end{itemize}
Similarly, a cycle in complex K-homology is represented by a triplet $(M,f,E)$, where $M$ is a $2n+i$-dimensional $\mathrm{spin}^c$ manifold with the boundary $N$, $f \colon (M,N) \to (X,Y)$ and $E$ is a complex vector bundle on $M$.
For equivalence relations between such triples, see \cite[Definition 5.6]{MR2330153} (see also \cite[Section 2]{MR2573141} for the $\KO$ case). 

The range of the relative assembly map is the $\K$-group (or $\KR$-group) of the relative group C*-algebra. Let $C^*_{\max}(\Gamma)$ denote the maximal group C*-algebra of $\Gamma$, that is, the completion of the group algebra $\bC[\Gamma]$ by the maximal C*-norm. It has the Real C*-algebra structure by the complex conjugation $\overline{\sum a_\gamma u_\gamma }:=\sum \overline{a}_\gamma u_\gamma$.  In this paper we use the same symbol $\phi$ for the induced $\ast$-homomorphism $C^*_{\max}(\Lambda) \to C^*_{\max}(\Gamma)$. 
\begin{defn}[{\cite[Section 2]{mathKT150603859}}]
Let $\phi \colon \Lambda \to \Gamma $ be a homomorphism of discrete groups. The maximal relative group C*-algebra is defined to be
\[
 C^*_{\max }(\Gamma, \Lambda):= S^{0,1} \otimes C\phi,
\]
where $S^{0,1}$ is the Real C*-algebra as in Notation \ref{notation}.
\end{defn}
Hereafter we omit the subscript ${\max}$ for simplicity of notations.

Let  
\[ \psi \colon SC^*(\Gamma) \to C\phi , \ \ \theta \colon C\phi \to C^*\Lambda \]
denote the inclusion $\psi(b_s):=(0,b_s)$ and the evaluation $\theta (a,b_s):=a$ respectively.
They induce the homomorphism of the Puppe exact sequence
\[ \cdots  \to  \KR_*(C^*\Lambda ) \xrightarrow{\phi _*  } \KR_*(C^* \Gamma ) \xrightarrow{\psi _*} \KR_*(C^*(\Gamma, \Lambda))\xrightarrow{\theta_*} \KR_{*-1}(C^*\Lambda )\to \cdots.
\]

The relative higher index is a map
\[ \mu_* = \mu ^{X,Y}_{*} \colon \KO_*(X,Y) \to \KR_*(C^*(\Gamma, \Lambda)) \]
which have the following properties.
\begin{enumerate}
\item[\eqnum \label{cond:funct}] It is functorial, that is, if we have a continuous map $\Phi \colon (X,Y) \to (X',Y')$ commuting with the reference maps, then $\mu ^{X',Y'} \circ \Phi_* = \mu ^{X,Y}$ holds. 
\item[\eqnum \label{cond:long}] The diagram 
\[
\xymatrix@C=0.7em{
\ar[r] &\KO_*(Y) \ar[r]^{i_*} \ar[d]^{\mu _{*}^Y} &\KO_*(X) \ar[r]^{j_*} \ar[d]^{\mu _*^{X}} & \KO_*(X,Y) \ar[r]^{\partial } \ar[d]^{\mu _*^{X,Y}} & \KO_{*-1}(Y) \ar[r]^{i_*} \ar[d]^{\mu_{*-1}^Y}&\KO_{*-1}(X) \ar[r] \ar[d]^{\mu_{*-1}^X} & \\
\ar[r]&\KR_*(C^*\Lambda )\ar[r]^{\phi _*  } &\KR_*(C^* \Gamma ) \ar[r]^{\psi _* \ \ \ \ } &\KR_*(C^*(\Gamma, \Lambda))\ar[r]^{\theta_*} &\KR_{*-1}(C^*\Lambda )\ar[r]^{\phi_*} &\KR_{*-1}(C^*\Gamma) \ar[r] &
}
\]
commutes.
\end{enumerate}
By the property (\ref{cond:funct}), it defines the relative assembly map 
\[
 \mu^{\Gamma, \Lambda}_* \colon \KO_*(B\Gamma, B\Lambda) \to \KR(C^*(\Gamma, \Lambda)), 
\]
where 
\[\KO_*(B\Gamma, B\Lambda):=\varinjlim _{(X,Y) \subset (B\Gamma , B\Lambda )} \KO_*(X,Y). \]
Hereafter we simply write the relative higher index map as $\mu_*$ or $\mu_*^{\Gamma, \Lambda}$ instead of $\mu_*^{X,Y}$.

\begin{rmk}\label{rmk:degree}
In this paper we focus on the higher index map $\mu_0^{\Gamma, \Lambda}$ of degree zero, defined for $\KO$-cycles represented by $8n$-dimensional spin manifolds. 
For other degrees, we define the higher index map $\mu^{\Gamma, \Lambda}_*$ by the composition
\begin{align*}
\KO_{-k}(B\Gamma  , B\Lambda ) \xrightarrow{\otimes [\bT^k]} & \KO_0(B(\Gamma \times \bT^k) ,B(\Lambda  \times \bT^k))\\
 \xrightarrow{\mu_0^{\Gamma \times \bZ^k, \Lambda \times \bZ^k}} & \KR_0(C^*(\Gamma \times \bZ^k, \Lambda \times \bZ^k))\\
 \xrightarrow{ \otimes_{C^*(\bZ^k)} \beta_{\bZ^k}} &\KR_{-k}(C^*(\Gamma, \Lambda)),
\end{align*}
where $\beta_{\bZ^k} \in \KKR_{-k}(C^*\bZ^k, \bR)$ is the top-degree element. That is,
\[  \mu_{-k}^{\Gamma, \Lambda }([M,f,E]):= \mu_0^{\Gamma \times \bZ^k, \Lambda \times \bZ^k }([M \times \bT^k, f \times \id_{\bT^k}, E ] ) \otimes_{C^*(\bZ)} \beta_{\bZ^k}. \]
The Kasparov product in the right hand side is taken through the identification $C^*(\Gamma \times \bZ^k, \Lambda \times \bZ^k) \cong C^*(\Gamma, \Lambda) \otimes C^*\bZ^k$.
\end{rmk}

\subsection{Coarse geometric definition by Chang--Weinberger--Yu}\label{section:2.1}
The first definition of the relative higher index is given by Chang--Weinberger--Yu~\cite{mathKT150603859} by using coarse index theory. 
Here we start with a brief summary of the coarse index theory. The basic references are \cite{MR1817560,MR1451759,MR2568691}.

Let $Z$ be a Riemannian manifold equipped with a regular Borel measure such that any non-zero function $f\in C_c(Z)$ acts on $L^2(Z)$ noncompactly. For an operator $T \in \bB(L^2(Z))$, its support $\supp T \subset Z^2$ is the complement of the union of $\supp f \times \supp g$, where $(f,g)$ runs over all pairs of compactly supported functions on $Z$ such that $gTf=0$. An operator $T \in \bB(L^2(Z))$ is said to be
\begin{itemize}
\item of finite propagation if $\mathop{\mathrm{Prop}} (T) := \sup \{ d(x,y) \mid (x,y) \in \supp T\}$ is finite,  
\item locally compact if $Tf, fT \in \bK( L^2(Z))$ for any $f \in C_c(Z)$.
\end{itemize}

For a Riemannian manifold $M$ with a $\Gamma$-Galois covering $\tilde{M}$ and $N \subset M$, we define four C*-algebras $C^*(\tilde{M})^\Gamma$, $C^*_L(M)$, $C^*_L(\tilde{M})^\Gamma$ and $C^*_L(N \subset M)$ as follows.
\begin{itemize}
\item The $\Gamma$-invariant maximal Roe algebra $C^*_{\max}(\tilde{M})^\Gamma $ is the closure of the Real $\ast$-algebra $\bC [\tilde{M}]^\Gamma$ of $\Gamma$-invariant locally compact operators on $L^2(\tilde{M})$ of finite propagation by the maximal norm $\| {\cdot}\|_{\max}$ among the C*-norms on it.  
\item The (resp.\ $\Gamma$-invariant) maximal localization algebra $C^*_{L, {\max}}(M)$ (resp.\ $C^*_{L, {\max}}(\tilde{M})^\Gamma )$ is the closure of the $\ast$-algebra $\bC_L[M]$ (resp.\ $\bC_L[\tilde{M}]^\Gamma $) of all uniformly continuous families $(T_t)_{t \in \bR_{\geq 0}}$ of (resp.\ $\Gamma$-invariant) locally compact operators on $M$ (resp.\ $\tilde{M}$) with $\mathop{\mathrm{Prop}} (T_t) \to 0$ as $t \to \infty$. 
\item For a subset $N \subset M$, we define the ideal $C^*_L(N \subset M)$ of $C^*_L(M)$ as the closure of the $\ast$-ideal $\bC_L [N \subset M]$ of $\bC_L[M]$ consisting of $T \in \bC_L[M]$ with $d(N^2, \supp T_t) \to 0$ as $t \to \infty$. 
\end{itemize}
These C*-algebras are equipped with the Real structure induced from that of the Hilbert space $L^2(\tilde{M})$.
Hereafter we omit the subscripts `max' for simplicity of notations.

Let $(M,f,E)$ be a representative of an element of $\KO _0(X,Y)$. We fix a Riemannian metric $g$ on $M_\infty$ whose restriction to $N_\infty$ is a product metric $dr^2 + g_N$. 
We introduce some homomorphisms between coarse C*-algebras arising from $(M_\infty , g)$. 
Note that an element $T \in \bC[\tilde{M}_1]^\Gamma$ is represented by a kernel function $T(x,y)$, which is a $\Gamma$-invariant Borel function on $\tilde{M} \times \tilde{M}$.
\begin{itemize}
\item Take a Borel section $M_2 \to \tilde{M}_2$ and let $\Sigma$ denote its image. Then we get a $\Gamma$-equivariant unitary $V \colon L^2(\tilde{M}_2) \to L^2(\Sigma) \otimes \ell^2 (\Gamma)$ which induces an isomorphism
\[ \zeta_{\tilde{M}_2}:=\Ad (V) \colon  C^*\Gamma \otimes \bK(L^2(\Sigma)) \to C^*(\tilde{M}_2)^\Gamma.  \]
\item Let $\bar{N}_2$ denote the restriction of the $\Gamma$-covering $\tilde{M}_2$ to $N_2$ and let $\bar{\pi} \colon \tilde{N}_2 \to \bar{N}_2$ denote the projection. The composition 
\[\tilde{\phi} := \zeta_{\bar{N}_2} \circ (\phi \otimes \id ) \circ \zeta_{\tilde{N}_2}\]
gives a $\ast$-homomorphism from $C^*(\tilde{N}_2)^\Lambda$ to $C^*(\bar{N}_2)^\Gamma$. The kernel function of $\tilde{\phi}(T)$ is written as
\[ \tilde{\phi}(T)(x,y) = \sum _{\bar{\pi} (\tilde{y})=y} T(\tilde{x}_0, \tilde{y}), \text{ a.e. $(x,y) \in \bar{N}_2 \times \bar{N}_2$} \]
where $\tilde{x}_0$ is an arbitrary choice of a point  of $\bar{\pi}^{-1}(x)$.  
Thanks to this description, we also get $\tilde{\phi}_L \colon C^*_L(\tilde{N}_2)^\Lambda \to C^*_L(\bar{N}_2)^\Gamma$ (cf.\  \cite[(2.13)]{mathKT150603859}).
\item We write the inclusions as 
\[ h \colon C^*(N_2) \to C^*(M_2),  \ \ \ \   h_L \colon C^*_L(N_2) \to C^*_L(M_2).\]
We also write the composition of similar inclusions of invariant Roe algebras with $\tilde{\phi}$ and $\tilde{\phi}_L$ as 
\[ \tilde{h} \colon C^*(\tilde{N}_2)^\Lambda \to C^*(\tilde{M}_2)^\Gamma , \ \ \ \  \tilde{h}_L \colon C^*_L(\tilde{N}_2)^\Lambda  \to C^*_L(\tilde{M}_2)^\Gamma .\]
We also write $h_L'$ for the inclusion $C^*_L(N_2' \subset M_2) \to C^*_L(M_2)$. 
\item The map $T_t \mapsto T_{t + \kappa}$ gives rise to an asymptotic morphism from $C^*(N_2' \subset M_2)$ to $C^*(N_2)$. Hence we get an asymptotic morphism 
\[ \tau _\kappa \colon Ch_L' \to Ch_L.\]  
\item Let $\pi \colon \tilde{M}_2 \to M_2$ denote the projection, let $R$ denote the injectivity radius of $M_2$ and let $0 < \varepsilon < R/2$. For $T \in \bB (L^2(M_2))$ with $\mathrm{Prop} (T) < \varepsilon$, the kernel function $T(x, y)$ uniquely lifts to a $\Gamma$-invariant function $\tilde{T}(x,y)$ on $\tilde{M}_2 \times \tilde{M}_2$ with the property that $\mathrm{Prop}(\tilde{T}) < \varepsilon$ and $\tilde{T}(x,y)=T(\pi (x),\pi(y))$ for any $x, y \in \tilde{M}_2$ with $d(x,y)<\varepsilon$. 
Now, $\theta_\kappa \colon T_t \mapsto \tilde{T}_{t+ \kappa }$ gives rise to an asymptotic morphism 
\[ \theta_\kappa \colon C^*_L(M_2) \to C^*_L(\tilde{M}_2)^\Gamma.\]
\end{itemize}

\begin{rmk}\label{rmk:misc}
We give some remarks on these C*-algebras and homomorphisms.
\begin{enumerate}
\item Let $D_\mathrm{alg} (\tilde{M}_r)^\Gamma$ denote the $\ast$-algebra of finite propagation operators $T$ on $L^2(\tilde{M}_r)$ which is pseudo-local, that is, $[T, f]$ is a compact operator for any $f \in C_c(\tilde{M}_r)$. It includes $\bC [\tilde{M}_r]^\Gamma$ as an algebraic ideal. It is proved in \cite[Lemma 2.16]{MR2568691} that an algebraic action $x \mapsto Tx$ is extended to a bounded adjointable endomorphism on $C^*(\tilde{M}_r)^\Gamma$. Hence $D_\mathrm{alg}(\tilde{M}_r)^\Gamma$ is identified with a subalgebra of $\cM(C^*(\tilde{M}_r)^\Gamma )$. 
\item Since the inclusion $C^*(\bar{N}_2)^\Gamma \to C^*(\tilde{M}_2)^\Gamma $ induces the isomorphism of $\KR$-groups, so is the $\ast$-homomorphism 
\[ \zeta_\phi  \colon C\phi \otimes \bK(L^2(\Sigma )) \to C\tilde{\phi}  \subset C\tilde{h}\]
given by $\zeta_\phi (a, b_s) = (\zeta _{\tilde{N}_2}(a), \zeta_{\bar{N}_2}(b_s))$.

\item By the construction of $\theta $ and $h_L$, the diagram 
\[
\xymatrix{
C^*_L (N_2) \ar[r]^{h_L} \ar[d]^{\theta_s^{N_2}} & C^*_L (M_2) \ar[d]^{\theta_s^{M_2}} \\ 
C^*_L (\tilde{N}_2) ^\Lambda \ar[r]^{\tilde{h}_L} & C^*_L(\tilde{M}_2)^\Gamma 
}
\]
commutes. Hence we get an asymptotic morphism 
\[ \chi _s \colon Ch_L \to C\tilde{h}_L.\]
\item Let $F$ be a $0$-th order pseudo-differential operator with $\mathrm{Prop}(F)<\varepsilon$. Then, its lift $\tilde{F}$ is also a $0$-th order pseudo-differential operator on $\tilde{M}_r$ and their principal symbols are related as $\sigma (\tilde{F})=\pi^* \sigma (F)$. 
\end{enumerate}
\end{rmk}

Finally we describe the definition of $\mu ^{\CWY}_0$. Let $S=S^0 \oplus S^1$ denote the spinor bundle on $M_\infty$ and let $S_E:=S \otimes E$. Let
\[D \colon C^\infty _c(M_{\infty}, S_E) \to C^\infty_c(M_\infty , S_E) \]
denote the Dirac operator on $M_\infty$ twisted by $E$. Since $M_\infty$ is a complete Riemannian manifold, there is a unique self-adjoint extension (see \cite[Theorem 1.17]{MR720933}) presented by the same letter $D$.
Choose an odd function $\chi $ on $\bR $ whose Fourier transform $\hat{\chi}$ is smooth at $x \neq 0$, $\supp \hat{\chi} \subset [-1,1] $ and $\hat{\chi}(x)=x^{-1}$ for $x \in [-1/2,1/2]$. Then, the operator
\begin{align} F_t := \chi (tD)= \frac{1}{2\pi } \int_{u \in \bR} \hat{\chi}(u)e^{iutD}du  \label{form:Ft}\end{align}
has propagation less than $t^{-1}$ by \cite[Proposition 2.2]{MR996446}. 

Let us fix a Borel isomorphism $S_E^0 \cong \bR^N_{M_r} \cong S_E^1$ to identify $L^2(M_r , S_E^0) \cong L^2(M_r )^{\oplus N} \cong L^2(M_r , S_E^1)$ as $\ast$-representations of $C_0(M_r)$. 
Then $F_t$ is regarded as an element of $\bM_{2N}(\cM(C^*_L(M_\infty )))$ by Remark \ref{rmk:misc} (1) and $(F_t^0)^*F_t^0 -1$, $F_t^0(F_t^0)^* -1$ are in $\bM_N(C^*_L(M))$, where $F_t^0$ is as in Notation \ref{notation}. That is, the image of $F_t^0$ in the Calkin algebra $\cQ(C^*_L(M_\infty))$ is a unitary.

Let $\ev _0 \colon C \tilde{h}_L \to C\tilde{h}$ denote the evaluating $\ast$-homomorphism at $t=0$, let
\begin{align*}
q_1 &\colon C^*_L(M_\infty ) \to C^*_L(M_\infty)/C^*_L(N_\infty) \\
q_2 &\colon Ch_L'  \to \frac{C h_L'}{C (C^*_L(N'_2 \subset M_2))} \cong S\frac{C^*_L(M_2)}{C^*_L (N'_2 \subset M_2)} 
\end{align*}
denote the quotients and let 
\[ j_1 \colon C^*_L(M_2)/C^*_L(N_2'\subset M_2) \to C^*_L(M_\infty)/C^*_L(N'_\infty \subset M_\infty ) \]
denote the isomorphism induced from the inclusion $C^*_L(M_2) \to C^*_L(M_\infty)$.

\begin{defn}[{\cite[Section 2]{mathKT150603859}}]
The Chang--Weinberger--Yu relative higher index $\mu ^\CWY_0$ is given by
\[ \mu ^\CWY _0([M,f,E]) := ((\zeta_\phi)_*^{-1} \circ \ev_0  \circ \chi \circ \tau \circ   (q_2)_*^{-1} \circ \beta \circ (j_1)_*^{-1} \circ (q_1)_*) (\partial [F_t^0]). \]
\end{defn}
Note that this homomorphism is independent of the choice of a regular Borel measure on $M$ such that any $f \in C_0(M)$ acts on $L^2(M)$ noncompactly. It is commented in \cite[Remark 3.9]{mathKT150701735} that $\mu^\CWY_0$ is well-defined as a homomorphism from the geometric K-homology group. It also follows from Theorem \ref{thm:equal} below.
According to \cite[Theorem 2.18]{mathKT150603859}, this is an obstruction for $(M,N)$ to admit a Riemannian metric with positive scalar curvature collared at $N$.

\subsection{Geometric definition by Deeley--Goffeng}\label{section:2.2}
In \cite{mathKT150701735}, Deeley and Goffeng introduce an alternative definition of the relative higher index within the framework of the geometric K-homology with coefficient in mapping cone C*-algebras \cite{MR3519049}. Here, the assembly map is defined as a homomorphism to the group $\K_*^\geom (\pt , C\phi)$, which is identified with $\K_*(C^*(\Gamma, \Lambda))$ by using the higher Atiyah--Patodi--Singer index theory (see for example \cite{MR1601842,MR1979016}). 

In summary, their relative higher index $\mu ^\DG_*$ is defined as following. Let $(M,f,E)$ be a representative of an element of $\KO_0(X,Y)$. Let 
\[ \cV := f^*(\tilde{X} \times_\Gamma C^*\Gamma)\]
be the Mishchenko line bundle over $M_\infty$. Let $S=S^0 \oplus S^1$ denote the spinor bundle of $M$ and set $S_{E,\cV}  := S \otimes E \otimes \cV $. Then, the twisted Dirac operator 
\[D_{\cV } \colon C_c^\infty (M_\infty ,  S_{E,\cV}  ) \to C_c^\infty(M_\infty , S_{E,\cV} )\]
is uniquely extended to an odd regular self-adjoint operator on the Hilbert $C^*(\Gamma)$-module $L^2(M_\infty , S_{E,\cV} )$ by \cite[Theorem 2.3]{MR3449594}. 

We also define the Mishchenko line bundle 
\[ \cW := (f|_N)^*(\tilde{Y} \times _\Lambda C^*\Lambda)\]
over $N$ and the odd Dirac operator 
 \[ \bD_{\cW} \colon C_c^\infty(N, S_{E,\cW }^0 ) \to C_c^\infty(N ,S_{E,\cW}^0 ),\]
which is uniquely extended to a regular self-adjoint operator on $L^2(N, S_{E,\cW} ^0 )$. 
Here the restriction of $S^0$ to $N$ is identified with the spinor bundle of $N$. Similarly we also define the odd Dirac operator $\bD_{\cV} $ acting on $L^2(N, S_{E,\cV}^0)$.
Note that the canonical isomorphism $\cW \otimes _{\phi} C^*\Gamma \cong \cV |_N$ identifies $\bD_{\cW } \otimes_\phi 1$ with $\bD_{\cV} $. 

As is shown in \cite[Theorem 3]{MR1979016} and \cite[Proposition 10]{MR1601842}, there is a smoothing operator $C \in \Psi ^{-\infty} (N, S_{E,\cV})$ in Mishchenko--Fomenko calculus such that $\bD_{\cV}  + C$ is invertible. Let $\eta$ be the smooth function supported on $M_1^\circ$ such that $0 \leq \eta \leq 1$ and $\eta \equiv 1 $ on $M$. Now, $\tilde{C}:=(1-\eta ) c(v)C$ determines an odd bounded operator on $L^2(M_\infty , S_{E,\cV} )$ and $D_\cV +\tilde{C}$ is a Fredholm operator. 
Set 
\[ F_{\cV, \tilde{C}} := \chi (D_\cV  + \tilde{C}) \in \bB(L^2(M_\infty , S_{E,\cV})). \]
Then the image of $F_{\cV, C}^0$ in the Calkin algebra $\cQ(L^2(M_\infty , S_{E,\cV}))$ is a real unitary and hence determine an element 
\[ [F_{\cV , C}^0] \in \KR_1(\cQ(L^2(M_\infty , S_{E,\cV}))).\]
Now we define the $b$-index as
\[\ind _b (D_\cV ,C):= \partial [F_{\cV, \tilde{C}}^0] \in \KR_0(\bK(L^2(M_\infty , S_{E,\cV}))) \cong \KR_0(C^*\Gamma ). \]

We also define the Hilbert $C\phi$-module bundle 
\[ \cX:=(f|_N)^*(\tilde{Y} \times _\Lambda C\phi) \]
over $N$, where $\gamma \in \Lambda$ acts on $C\phi$ by multiplication of $(u_\gamma, u_{\phi(\gamma)}) \in \cM(C\phi)$ from the left. Then the compact operator algebra $\bK(L^2(N, S_{E, \cX}))$ is canonically isomorphic to the mapping cone $\dot{\phi}$, where 
\[ \dot{\phi}:= {\cdot} \otimes _\phi 1 \colon \bK(L^2(N ,  S_{E, \cW})) \to \bK(L^2(N, S_{E,\cV} )). \]
Therefore, $C\dot{\phi}$ has the same $\KR$-groups with $C\phi$. 

Let $\bD_{\cV, C}(s) := (1-s)^{-1}(\bD_\cV + sC)$ for $s \in [0,1)$ and let $\Cay(x)=\frac{ix+1}{ix-1}$ denote the Cayley transform. Then $(\Cay (\bD_\cW ), \Cay(\bD_{\cV, C}(s))$ is a transpose-invariant unitary of $(C\dot{\phi})^+$, which determines an element of $\KR_{-1}(C\dot{\phi})$ by Remark \ref{rmk:KR-1}. Now we define
\[ c(\bD_\cW , C) := [(\Cay (\bD_\cW ), \Cay(\bD_{\cV, C}(s))] \in \KR_{-1}(C \dot{\phi} ) \cong \KR_{-1}(C\phi). \]
\begin{defn}[{\cite[Theorem 4.9]{mathKT150701735}}]
The Deeley--Goffeng relative assembly map is defined to be
\[\mu^\DG_0 ([M,f,E])= \psi_* (\ind _b (D_{\cV} , C)) + c(\bD_\cW , C). \]
\end{defn}

\section{The relative Mishchenko--Fomenko higher index}\label{section:3}
In this section we give a new definition of the relative higher index It is thought of as a relative version of the Mishchenko--Fomenko higher index~\cite{MR548506} in the sense that it is realized as a Kasparov product in KK-theory. 

Let $(X,Y)$ be a pair of finite CW-complexes with a reference map $f \colon (X,Y) \to (B\Gamma , B\Lambda)$. Throughout this section, we use the notations $X_r$, $Y_r$ and $Y'_r$ introduced in Notation \ref{not:Xr}.
As in Subsection \ref{section:2.2},  let $\cV  := \tilde{X}\times _\Gamma C^*\Gamma$, $\cW :=\tilde{Y} \times _\Lambda C^*\Lambda$ and $\cX := \tilde{Y} \times _\Lambda C\phi$.  Note that the inclusion $\psi \colon SC^*\Gamma \to C\phi$ induces a fiberwise inclusion of bundles $\psi_Y \colon S\cV|_Y \to \cX$. 
Set
\begin{align}
\begin{split}
\sE_r &:= SC(X, \cV) \oplus _{C(Y, \cX)}C_0(Y \times [0,r) , \cX ) \\
&=\{ (\xi ,\eta ) \in C(X, S\cV) \oplus C_0(Y \times [0,r) , \cX ) \mid \psi_Y(\xi |_{Y}) = \eta|_{Y \times \{ 0 \} }  \}.
\end{split}
\label{form:E}
\end{align}
In other words, $\sE_r$ is the section space $C_0(X_r^\circ , \cE)$ of a bundle (more precisely, an upper semi-continuous field in the sense of \cite[Definition A.1]{MR2119241}) of Hilbert $C\phi$-modules $\cE$ over $X_r$ constructed by clutching $\cX \to Y_r$ with $S\cV \to X$ at $Y$ by $\psi_Y$.

Let 
\begin{align} 
\rho (r,s) = \rho_s(r) := \min \{ 1 , 2s+2r-3 \} \in C([1,2] \times [0,1]). \label{form:rho}\end{align}
We regard it as a continuous function on $X_\infty \times [0,1]$ by $\rho(x,s):=2s-1$ for $x \in X_1$, $\rho(y,r,s):=\rho(r,s)$ for $(r,y) \in Y_2' $ and $\rho(y,r,s)=1$ for $r>2$.
\begin{figure}[t]
\centering 
\begin{tikzpicture}[scale=0.8]
\shade [top color = black!90!white, bottom color = black!90!white, middle color = white, shading angle = 135] (-2,-2) rectangle (4,4);
\shade [top color = black!42!white, bottom color = black!42!white, middle color = white] (-2,0) rectangle (0,4);
\fill [color = black!42!white] (4,0) -- (4,4) -- (0,4) -- (4,0);
\fill [color = white] (-2.1,-2.1) rectangle (4.1,0);
\draw [->] (-3,-0.3) -- (-3,4.5);
\draw [->] (-1,-1) -- (4.5,-1);
\draw [thick,dashed] (-2,4) -- (4,4);
\draw [thick, dashed] (4,4) -- (4,0);
\draw [thick] (4,0) -- (0,0);
\draw (0,0) -- (0,4);
\draw [thick,dashed] (0,0) -- (-2,0);
\node at (4.3,-0.7) {$r$};
\node at (-2.2,4.3) {$s$};
\node at (-1,4.5) {$X_1$};
\node at (2,4.5) {$Y_2'$};
\fill (4,-1) coordinate (O) circle[radius=1pt] node[below=2pt]  {$2$};
\fill (0,-1) coordinate (O) circle[radius=1pt] node[below=2pt]  {$1$};
\fill (-3,4) coordinate (O) circle[radius=1pt] node[left=2pt]  {$1$};
\fill (-3,0) coordinate (O) circle[radius=1pt] node[left=2pt]  {$0$};
\end{tikzpicture}
\caption{The shading shows the value of $|\rho (r,s)|$.} \label{fig:1}
\end{figure}
Then the multiplication by $\rho$, that is,
\[ \rho \cdot  (\xi, \eta) = (\rho |_X \cdot \xi , \rho|_{Y \times [0,r)} \cdot \eta ) \text{ for $(\xi,\eta) \in \sE_r$},   \]
determines a self-adjoint element $\rho \in \bB (\sE_r  )$. It satisfies $\rho ^2-1 \in \bK(\sE_r )$ since the function $\rho^2-1$ vanishes at $\{s=1\}$, $\{ s=0 \} \cap \{ r\leq 1 \} $ and $\{r>2 \}$ (the dotted lines in Figure \ref{fig:1}). 
Now we define the \emph{relative Mishchenko line bundle} $\ell_{\Gamma, \Lambda}$ as 
\[ \ell_{\Gamma, \Lambda}=\ell_{\Gamma, \Lambda}^{X,Y} :=[ \sE_\infty , 1, \rho ] \in \KKR_{-1}(\bR , C_0(X_\infty ) \otimes C\phi).\]
Recall that the group $\KKR_{-1}(\bR , C_0(X_\infty ) \otimes C\phi)$ is canonically isomorphic to $\KKR (\bR , C_0(X^\circ ) \otimes C^*(\Gamma, \Lambda))$.

\begin{defn}
The \emph{relative Mishchenko--Fomenko higher index} $\mu^\MF_*$ is defined by the Kasparov product
\[ \ell_{\Gamma, \Lambda } \hotimes_{C_0(X^\circ)} { \cdot }\  \colon \KKR_*(C_0(X^\circ ), \bR ) \to \KR_*(C^*(\Gamma, \Lambda )). \] 
\end{defn}
We also use the symbol $\alpha_{\Gamma, \Lambda}$ or $\alpha_{\Gamma, \Lambda}^{X,Y}$ for the homomorphism $\mu_*^\MF$.

\begin{lem}
The Mishchenko--Fomenko relative higher index $\mu^\MF_*$ satisfies (\ref{cond:funct}) and (\ref{cond:long}) in page \pageref{cond:funct}. 
\end{lem}
\begin{proof}
Assume that we have a continuous map $\Phi \colon (X,Y) \to (X',Y')$. Let us extend it to a continuous map $\Phi_\infty \colon X_\infty \to X_\infty'$. By definition we have $\Phi_\infty^*\ell_{\Gamma, \Lambda}^{X',Y'}=\ell_{\Gamma, \Lambda}^{X,Y}$, which implies the functoriality (\ref{cond:funct}).

To see (\ref{cond:long}), recall that the higher index map $\mu_*^\Gamma$ is given by the Kasparov product with $\ell_\Gamma := [C(X, \cV),1, 0]$ (see \cite{MR3359022}). Then the commutativity of the diagram follows from 
\begin{align*}
(\beta \otimes \ell_\Gamma) \hotimes _{SC^*(\Gamma )} [\psi] &= [SC(X,\cV), 1 , 2s-1]= \ell_{\Gamma ,\Lambda} \hotimes _{C_0(X^\circ _1)} [\iota ], \\
\ell_{\Gamma, \Lambda } \hotimes _{C \phi } [\theta ] &= [ C_0(Y_2^\circ ,\cW), 1 , \rho_0(r)]=  (\ell_\Lambda \otimes \beta ) \hotimes _{C_0(Y_1^\circ)} [i^*], 
\end{align*}
where $\iota \colon C_0(X_1^\circ) \to C(X)$ denote the restriction and $i^* \colon C_0(Y_1^\circ) \to C_0(X_1^\circ)$ denote the open embedding. Note that in the second equality we canonically identify $C_0(Y_2^\circ)$ with the suspension $SC(Y)$.
\end{proof}

\begin{defn}\label{defn:dBC}
We call the Kasparov product
\begin{align*}
\beta_{\Gamma ,\Lambda}^{X,Y} := \ell_{\Gamma, \Lambda} \hotimes_{C^*(\Gamma, \Lambda)} { \cdot } \colon \KKR (C^*(\Gamma, \Lambda) , \bR) \to \KR^*(X,Y) \label{form:dBC}
\end{align*}
the \emph{dual relative higher index map}.
\end{defn}
This map has functoriality dual to (\ref{cond:funct}). Hence if $(B\Gamma, B\Lambda)$ has the homotopy type of a pair of finite CW-complexes, we have $\beta_{\Gamma, \Lambda}^{X,Y} = \Phi^* \circ \beta_{\Gamma, \Lambda}^{B\Gamma , B\Lambda }$ for any pair $(X,Y)$ of finite CW-complexes with the reference map $\Phi \colon (X,Y) \to (B\Gamma, B\Lambda)$.

\begin{prp}
For any $x \in \KR_*(X,Y)$ and $\xi \in \KR^*(C^*(\Gamma, \Lambda))$, we have
\[ \langle \alpha_{\Gamma, \Lambda}(x) , \xi \rangle = \langle x, \beta_{\Gamma, \Lambda}(\xi ) \rangle \in \KKR(\bR,\bR) \cong \bZ   \]
\end{prp}
\begin{proof}
It immediately follows from the associativity of the Kasparov product.
\end{proof}

We give a more explicit description of $\mu^\MF_0$. Let $(M,f,E)$ be a representative of an element of $\KO_0(X,Y)$. Then the corresponding element of the analytic $\KO$-homology cycle is represented by the Kasparov bimodule
\[ [ L^2 (M_\infty, S_E), f^* , F:=\chi(D) ], \] 
where $ D$ is the Dirac operator on $M_\infty$ twisted by $E$ and $\chi$ is as in (\ref{form:Ft}).

The compact operator algebra on the Hilbert $C\phi$-module 
\[ L^2(M_\infty, S _\cE) := \sE \otimes_{C_0(M_\infty)} L^2(M_\infty, S) \]
is isomorphic to the mapping cone $C\bar{\phi}$, where $\bar{\phi}$ is the composition
\[ \bK(L^2(N_\infty , S_{E, \cW })) \xrightarrow{{\cdot} \otimes _\phi 1} \bK(L^2(N_\infty , S_{E, \cV})) \subset \bK(L^2(M_\infty , S_{E, \cV})). \]

Let $D_\cV$ denote the Dirac operator on $M_\infty$ twisted by $\cV \otimes E$ and let $D_\cW$ denote the Dirac operator on  $N \times \bR$ twisted by $\cW \otimes E$, which are extended to regular self-adjoint operators on the $L^2$-spaces. Let $\chi$ be as in (\ref{form:Ft}) and set
\begin{align*}
F_\cW &:= \chi (D_\cW) \in \bB(L^2(N \times \bR , S_{E, \cW}) ), \\
F_\cV &:= \chi (D_\cV) \in \bB(L^2(M_\infty , S_{E, \cV})) .
\end{align*}
Let $V \colon L^2(N  \times \bR , S_{E, \cV}) \to L^2(M_\infty , S_{E, \cV} )$ denote the partial isometry identifying subspaces $L^2(N_\infty, S_\cV)$ of the domain and the range. 
\begin{lem}\label{lem:cptsupp}
The difference $F_\cV - V(F_\cW \otimes _\phi 1)V^*$ is in $\bB(L^2(M_1 , S_{E, \cV}))$. 
\end{lem}
\begin{proof}
This is essentially proved in \cite[Proposition 1.5]{MR1120996}. The Fourier transform presentation (\ref{form:Ft}) of $\chi(D_\cV)$ implies that the restriction of $\chi(D_\cV)s$ to $N \times [1,\infty)$ depends only on the restriction of $D_\cV$ and $s$ to $N \times [0, \infty)$ (see \cite[Lemma 2.1, Proposition 2.2]{MR996446}). In particular, we get
\[ \chi(D_\cV)s = V(\chi(D_\cW) \otimes _\phi 1)V^*s\] 
for $s \in C^\infty_c(M_\infty , S_{\cV,E})$ with $\mathrm{supp} (s) \subset N \times [1,\infty)$. 
\end{proof}
 
Let $\sigma_s := (1-\rho_s^2)^{1/4}$. Since $\sigma_0$ is supported on $N \times [1,2]$, we obtain that  ${\sigma_0 F_\cW \sigma _0 \otimes _\phi 1}= \sigma_0 F_\cV \sigma_0$. Hence, by Lemma \ref{lem:cptsupp}, we have 
\[ (\sigma _0 F_\cW \sigma_0, \sigma _s F_\cV \sigma _s) \in \cM (C\bar{\phi}) \cong \bB(L^2(M_2 , S_{E, \cE})). \]

\begin{prp}
Through the isomorphism of $\KO _0(X,Y)$ with the analytic $\K$-homology group $\KKR (C_0(X^\circ ), \bR )$, the homomorphism $\mu ^\MF_*$ maps $[M,f,E]$ to $\partial [(T_\cW , T_\cV(s))] \in \KR_{-1}(C\bar{\phi})$, where 
\[ 
T_{\cW}=  \begin{pmatrix}
\rho_0 & \sigma_0 F_\cW^1 \sigma_0 \\ \sigma_0 F_\cW^0 \sigma_0 & -\rho_0
\end{pmatrix}_{\textstyle ,}  \ \ T_{\cV}(s):= \begin{pmatrix}
\rho_s & \sigma_s F_\cV^1 \sigma_s \\ \sigma_s F_\cV^0 \sigma_s & -\rho_s
\end{pmatrix}_{\textstyle .} \]
\end{prp}
\begin{proof}
According to \cite[Proposition 18.3.3]{MR1656031}, there always exists an $F$-connection 
\[ (G_\cW,G_\cV(s)) \in \cM(C\bar{\phi}) \cong \bB(L^2(M_\infty, S_{E , \cE } ) ). \]
Note that $G_{\cW} \in \bB(L^2(N_\infty , S_{E, \cW}))$ and each $G_{\cV}(s) \in \bB(L^2(M_\infty, S_{E,\cV}))$ are $ F $-connections. Recall that $\sigma _0 F_\cW \sigma _0$ (resp.\ $ \sigma _s F_\cV \sigma _s$) is a $\sigma_0F \sigma _0$- (resp.\ $\sigma _s F \sigma _s$-) connection. By \cite[Example 18.3.2(c)]{MR1656031}, we obtain that 
\begin{align*}
\sigma_0 F_\cW \sigma_0 - \sigma _0 G_\cW \sigma _0 & \in \bK(L^2(N_\infty, S_{E,\cW})), \\
\sigma _s F_\cV \sigma _s - \sigma _s G_\cV(s) \sigma _s & \in \bK(L^2(M_\infty , S_{E, \cV})).
\end{align*} 

Now, the operator
\[ \bigg( \begin{pmatrix}
\rho_0 & \sigma_0 G_\cW^1  \sigma_0 \\ \sigma_0G_\cW^0 \sigma_0 & -\rho_0
\end{pmatrix}_{\textstyle ,} \begin{pmatrix}
\rho_s & \sigma_sG_\cV (s)^1 \sigma_s\\ \sigma_sG_\cV (s)^0\sigma_s & -\rho_s
\end{pmatrix} \bigg) \]
is a compact perturbation of $(T_\cW, T_\cV(s))$ and hence is a Fredholm operator on $L^2(M_\infty ,S_{E,\cE})$. We apply Lemma \ref{lem:Kas} to complete the proof. 
\end{proof}

\section{Coincidence of the relative higher index maps}\label{section:3.5}
The goal of this section is to prove the following.
\begin{thm}\label{thm:equal}
The homomorphisms $\mu^\CWY_*$, $\mu_*^\DG$ and $\mu_*^\MF$ coincides.
\end{thm}
Hereafter we only show Theorem \ref{thm:equal} for degree $0$ case, namely $\mu^\CWY_0 = \mu_0^\MF =\mu_0^\MF$. As is commented in Remark \ref{rmk:degree}, it is enough for the proof of the full statement of Theorem \ref{thm:equal}.

\subsection{The proof of $\mu^{\mathrm{CWY}}_*=\mu^{\mathrm{MF}}_*$}
As in Subsection \ref{section:2.1}, we fix a Borel isomorphism between $S_E^0$, $S_E^1$ and the trivial bundle $\underline{\bR}_{M_\infty}^N$ and identify the $\ast$-representations $L^2(M_\infty, S_E^0) \cong L^2(M_\infty)^{\oplus N} \cong L^2(M_\infty, S_E^1)$ of $C_0(M_\infty)$.
\subsubsection*{Step 1.}
Let $F_t$ be as in (\ref{form:Ft}). For $s \in [0,1]$, let $\tau_s :=(1-(2s-1)^2)^{1/4}$. The continuous function
\[ s \mapsto T'_t(s):=\begin{pmatrix} 2s-1 & \tau_s F^*_t \tau_s \\ \tau_s F_t \tau_s & 1-2s \end{pmatrix} \in \bM_{2N} (\cM(C^*_L(M_\infty )) )
\]
on $[0,1]$ satisfies that $T'_t(0)^2= T'_t(1)^2 = 1$. Hence $T'_t$ is a real self-adjoint element $T_t' \in \bM_{2N}(\cM(SC^*_L(M_\infty)))$ with $(T_t')^2-1 \in \bM_{2N}(SC^*_L(M_\infty))$, that is, it determines an $\KR$-class
\[ [T_t'] \in \KR_{0}(\cQ(SC^*_L(M_\infty)))\]
and hence $\partial [T_t'] \in \KR_{-1}(SC^*_L(M_\infty))$.
\begin{lem}\label{lem:step1}
We have $\beta \otimes \partial [F_t^0] = \partial [T_t'] \in \KR_{-1}(SC^*_L(M_\infty))$.
\end{lem}
\begin{proof}
According to Remark \ref{rmk:KR-1}, $\partial [T_t']$ corresponds to the Real self-adjoint $\KK$-class 
\[ [SC^*_L(M_\infty)^{\oplus 2N}, 1 , T_t' ] \in \KKR_{-1}(\bR, SC^*_L(M_\infty)). \]
By Lemma \ref{lem:Kas}, it coincides with the Kasparov product of the Bott generator $\beta :=[S, 1, 2s-1 ] \in \KKR_{-1}(\bR, S)$ with
\[ \bigg[ C^*_L(M_\infty )^{\oplus N} \oplus (C^*_L(M_\infty)^{\mathrm{op}})^{\oplus N} , 1 ,  \begin{pmatrix} 0 & F_t^1 \\ F_t^0 & 0 \end{pmatrix} \bigg] \in \KKR ( \bR,  C^*_L(M_\infty)), \]
which corresponds to $\partial [F_0]$ through the isomorphism $\KR_0(C^*_L(M_\infty)) \cong \KR_{-1}(\cQ(C^*_L(M_\infty ))) \cong \KKR(\bR, C^*_L(M_\infty)) $. This shows that $\partial[F_t^0]$ is mapped by the Bott isomorphism to $\partial [T_t'] $. 
\end{proof}

\subsubsection*{Step 2.}
Let $\rho_s \in C_b(M_\infty)$ be as in (\ref{form:rho}) and let $\sigma _s:=(1-\rho_s^2)^{1/4}$.
Note that we have $\sigma _0 F_t^0 \sigma_0 \in D_{\mathrm{alg}}(N_2')$ and $\sigma _s F_t^0 \sigma_s \in D_{\mathrm{alg}}(M_2) $. Hence, by Remark \ref{rmk:misc} (1), the function
\[ s \mapsto  T_t(s):=\begin{pmatrix} \rho_s & \sigma_s F^1_t \sigma_s \\ \sigma_s F_t^0 \sigma_s  & -\rho_s \end{pmatrix}  \]
on $[0,1]$ is a real self-adjoint element of $\bM_{2N}(\cM(Ch_L))$ satisfying $T_t(0)^2-1 \in C^*(N_2')$ and $T_t(1)^2=1$.

Let $h_L'$, $q_1$, $q_2$, $\pi$ and $j_1$ be as in Subsection \ref{section:2.1}. 
\begin{lem} \label{lem:step2}
Then we have
\[ ((q_2)*^{-1} \circ \beta \circ (j_1)_*^{-1} \circ \pi_*)(\partial [F_t^0])=\partial [T_t] \in \KR_{-1}(Ch_L).\] 
\end{lem}
\begin{proof}
Let $k_L' \colon C^*_L(N_\infty' \subset M_\infty) \to C^*_L(M_\infty)$ and $\iota \colon SC_L^*(M_\infty) \to Ck_L'$ denote the inclusions, let $q_3 \colon Ck_L \to S\frac{C^*_L(M_\infty)}{C^*(N_\infty' \subset M_\infty)}$ denote the quotient and let $j_2 \colon Ch_L' \to Ck_L'$ denote the inclusion induced from $M_2 \subset M_\infty$. Then, the diagram
\[\xymatrix@C=3em{
\KR_ {-1}(SC^*_L(M_\infty )) \ar[rd]_{(q_1 \otimes \id_S)_*} \ar[r]^{\iota_*} & \KR_ {-1}(Ck_L')    \ar[d]^{(q_3)_*} _\cong  &  \KR_ {-1}(Ch_L') \ar[l]_{(j_2)_*} \ar[d]^{(q_2)_*}_\cong \\
& \KR_ {-1}\Big( S\frac{C^*_L(M_\infty)}{C^*_L(N_\infty \subset M_\infty )} \Big)  &\ar[l]_{(j_1 \otimes \id_S)_*}^\cong  \KR_{-1}\Big( S\frac{C^*_L(M_2)}{C^*_L(N_2' \subset M_2)} \Big)
}
\]
commutes. 

Let $\bar{\rho}_{\kappa ,s}(r):= \kappa  \rho_s(r) + (1-\kappa )(2s-1)$ and let $\bar{\sigma}_{\kappa ,s}:= (1-\bar{\rho}_{\kappa ,s}^2)^{1/4}$. Then,
\[ \bar{T}_{t}^\kappa (s):=\begin{pmatrix} \bar{\rho}_{\kappa ,s} & \bar{\sigma}_{\kappa ,s} F^1_t \bar{\sigma}_{\kappa ,s} \\ \bar{\sigma}_{\kappa ,s} F_t^0 \bar{\sigma}_{\kappa ,s}  & -\bar{\rho}_{\kappa ,s} \end{pmatrix} \]
is a homotopy connecting $j_2(T_t)$ and $\iota(T'_t)$ in the set of real self-adjoint unitaries of $\cQ (Ck_L)$. That is, $(\tilde{j}_2)_*(\partial [T_t ])=\iota_*(\partial [T'_t])$ holds in $\KR_{-1}(Ck_L)$. Since $(j_2)_*$ is an isomorphism, we get
\begin{align*}
& ((q_2)_*^{-1} \circ \beta \circ (j_1)_*^{-1} \circ (q_1)_* )(\partial [F_t]) \\
= & ((q_2)_*^{-1} \circ (j_1 \otimes \id_S)_*^{-1} \circ (q_1 \otimes \id_S)_*)(\partial [T_t']) \\
= & ((j_2)_*^{-1} \circ \iota_*)( \partial [T'_t]) = \partial [T_t]. \qedhere
\end{align*}
\end{proof}

\subsubsection*{Step 3}
Let $t>0$ be such that 
\begin{itemize}
\item $T_t(0) \in \bM_{2N}(\cM(C^*(N_1)))$ lifts to $\breve{T} \in \bM_{2N}(\cM(C^*(\tilde{N}_2)^\Lambda))$ and
\item each $T_t(s) \in \bM_{2N}(\cM(C^*(M_1)))$ lifts to $\tilde{T}(s) \in \bM_{2N}(\cM(C^*(\tilde{M}_2)^\Gamma))$.
\end{itemize}
By Remark \ref{rmk:misc} (3), we have $\tilde{h}_L (\breve{T}_t)= \tilde{T}(0)$ and hence $(\breve{T}, \tilde{T}(s)) \in \bM_{2N}(\cM(C\tilde{h}_L))$. Consequently we obtain the following.
\begin{lem}\label{lem:step3}
We have
\[ \mu_0^\CWY ([M,\id, E])=\partial [(\breve{T}, \tilde{T}(s))] \in \KR_{-1}(C\tilde{\phi}) . \]
\end{lem}

\subsubsection*{Step 4}
Finally here we prove $\mu^{\CWY}_0=\mu^{\MF}_0$. To this end, we recall the isomorphism $\bK(L^2(M_2, S_{E, \cV })) \cong C^*(\tilde{M}_2)^\Gamma$ following to the idea given in \cite[Lemma 2.3]{MR1909514}.  

\begin{lem}\label{lem:completion}
Let $\tilde{D}$ denote the Dirac operator on $\tilde{M}_\infty$. Then there is a $\ast$-isomorphism $\varphi \colon\bK(L^2(M_2, S_{E, \cV })) \to C^*(\tilde{M}_2)^\Gamma $ satisfying
\[ \varphi (\sigma _s F_\cV \sigma_s) = \sigma _s \chi (\tilde{D}) \sigma _s \in \cM(C^*(\tilde{M}_2)^\Gamma ). \]
\end{lem}
\begin{proof}
First we construct the $\ast$-isomorphism $\varphi$.  
Let $\cU:= \{ U_\mu\}_{\mu \in I}$ be a finite open cover of $X$ such that the restriction of $\tilde{X}$ to each $U_\mu$ is a trivial bundle and let $\gamma _{\mu \nu}$ denote the transformation function. 
Let $\{ \eta_\mu \}_{\mu \in I}$ be a family of continuous functions such that $\mathrm{supp} (\eta_\mu) \subset U_\mu$, $0 \leq \eta_\mu(x) \leq 1$ and $\sum \eta_\mu ^2 =1$. We write $\bM_I$ for the matrix algebra on $\bC^I$ and let $\{ e_{\mu \nu}\}_{\mu , \nu \in I}$ denote the matrix unit. Then, 
\begin{align*}
P_\cV := \sum_{\mu , \nu \in I} \eta_\mu \eta_\nu \otimes u_{\gamma_{\mu \nu}} \otimes e_{\mu \nu} \in C(X) \otimes_{\mathrm{alg}} \bR [\Gamma] \otimes \bM_I \label{form:projMF}
\end{align*}
is a projection whose support is isomorphic to $\cV$ as Hilbert $C^*\Gamma$-module bundles on $X$. Now we define a dense $\ast$-subalgebra
\begin{align*}
\bK_{\mathrm{alg}}(L^2(M_2, S_{E, \cV })) &:= P_\cV (\bK(L^2(M_2, S_E)) \otimes_{\mathrm{alg}} \bR [\Gamma ] \otimes \bM_I)P_\cV 
\end{align*}
of $\bK(L^2(M_2, S_{E, \cV }))$. Let $\lambda$ denote the left regular representation of $\Gamma$. 
The bundle isomorphism $\cV \otimes_\lambda \ell^2(\Gamma) \cong \tilde{M}_r \times _{\lambda } \ell^2(\Gamma )$ induces a unitary isomorphism 
\[ U \colon L^2(M_r , S_{E, \cV }) \otimes_{\lambda }\ell^2(\Gamma) \to L^2(\tilde{M}_r , S_E) \] 
for any $r \in [0,\infty]$. Then,  $\varphi(x):= U(x \otimes _\lambda 1)U^*$ gives a bijection 
\[ \varphi \colon \bK_{\mathrm{alg}}(L^2(M_2, S_{E, \cV })) \to \bR [\tilde{M}_2]^\Gamma, \]
which extends to a $\ast$-isomorphism $\varphi \colon \bK(L^2(M_2, S_{E, \cV })) \to C^*(\tilde{M}_2)^\Gamma $ by the universality of both sides. 
It also induces the $\ast$-isomorphism between their multiplier algebras.

Since $U$ comes from an isomorphism of flat bundles of Hilbert spaces, we have $U(D_{\cV} \otimes _\lambda 1)U^*$ coincides with $\tilde{D}$ on $C_c^\infty(\tilde{M}_\infty , S_E)$. By the uniqueness of self-adjoint extension of $\tilde{D}$ shown in \cite[Theorem 1.17]{MR720933}, they coincide as closed self-adjoint operators. That is, 
\[ U(\chi(D_{\cV}) \otimes _\lambda 1)U^* = \chi(\tilde{D}) \in D^*_{\mathrm{alg}}(\tilde{M}_\infty)^\Gamma \]
holds. Now the lemma is proved by multiplying with $\sigma_s$ from both sides. 
\end{proof}

Now, $\tilde{T}(s)$ and 
\[ \varphi(T_\cV(s))= \begin{pmatrix}\rho_s & \sigma _s \chi (\tilde{D})^1 \sigma _s  \\ \sigma _s \chi (\tilde{D})^0 \sigma _s  & -\rho_s \end{pmatrix}\]
are $0$-th order pseudo-differential operators on $\tilde{M}$ with the same principal symbol by Remark \ref{rmk:misc} (4). Hence we get $\tilde{T}(s) - \varphi (T_\cV(s)) \in \bM_{2N}(C^*(\tilde{M})^\Gamma )$. Similarly we also get $\breve{T} - \varphi (T_\cW) \in \bM_{2N}(C^*(\tilde{N})^\Lambda )$. This concludes the proof of $\mu^{\CWY}_0=\mu^{\MF}_0$.

\subsection{The proof of $\mu^{\mathrm{DG}}_*=\mu^{\mathrm{MF}}_*$}
\subsubsection*{Step 1}
We apply the theory of noncommutative spectral section developed in \cite{MR1979016} to choose a useful perturbation of $\bD_\cV$.

\begin{lem}\label{lem:APS}
There exists a smoothing operator $C \in \Psi^{-\infty}(N, S_{E, \cV })$ such that $\bD_\cV + C$ is invertible and $\ind_b(D_\cV , C)=0$.
\end{lem}
Before the proof, we recall the $\K$-theory class determined by a difference of infinite projections. For a pair of projections $(P,Q) \in \bB(L^2(N, S_{E, \cV}))$ such that $P-Q$ is compact, we define the difference class $[P-Q] \in \K_*(C^*\Gamma ) $ as in \cite[Remark 1]{MR733641}. Note that, if $Q=P-p+q$ by compact projections $p \leq P$ and $q \leq 1-P$, then $[P-Q]=[p]-[q]$. 
\begin{proof}
Let $\cP$ be a noncommutative spectral section of $\bD_\cV$ (in the sense of \cite[Definition 3]{MR1979016}). By \cite[Proposition 2.10]{MR1601842}, there is a self-adjoint smoothing operator $C_\cP \in \Psi^\infty (N, S_{E , \cV})$ such that $\bD_\cV + C_\cP$ is invertible and its positive spectral projection is $\cP$. 
By the spectral flow formula
\[ \ind _{\APS} (D_\cV , \cQ) - \ind_{\APS}(D_\cV , \cP )= [\cP - \cQ ]  \]
shown in \cite[Theorem 5]{MR1979016} and 
\[\ind_b(D_\cV, C_\cP)= \ind_\APS (D_\cV , \cP) \]
shown in \cite[Theorem 6]{MR1979016}, it suffices to show that there is another spectral section $\cQ$ such that $[\cP - \cQ]=-\ind_{\APS}(D_\cV , \cP)$. Indeed, $C_\cQ$ is the desired smoothing operator.

It is shown in (the proof of) \cite[Theorem 2]{MR1979016} that $\cP$ and $1-\cP$ are full projections in $\bB(L^2(N, S_{E , \cV}))$. 
Hence there are compact projections $p \leq 1-\cP$ and $q \leq \cP$ such that $[p]-[q]= \ind _b (D_\cV ,C) $. 
Let $\cQ':=\cP + p -q$. 
The proof of \cite[Theorem 3]{MR1979016} actually claims that, for any projection $\cQ'$ such that $\cQ'  - \bD_\cV(1+\bD_\cV^2)^{-1/2}$ is compact, there is a noncommutative spectral section $\cQ$ of $\bD_\cV$ such that $\| \cQ'- \cQ \| < 1/2$ and hence $[\cQ' - \cQ ]=0$. Consequently we get $[\cQ - \cP ] = [\cQ' - \cP] = [p]-[q]=\ind_{\APS }(D, \cP)$.
\end{proof}

\subsubsection*{Step 2}
In this step we provide a slightly tricky replacement of a representative of $\mu_0^{\mathrm{MF}}([M,\id ,E])$ in the way that it is supported on the cylinder $N \times [0,\infty)$. 
For a continuous function $f \in C_b(M_\infty)$ and $s\in [0,1)$, let $f^s$ denote the function given by $f^s|_X=f|_X$ and $f(y,r)=f(y,sr)$.

Let $C \in \Psi^{-\infty}(N, S_{E , \cV})$ be a smoothing operator as in Lemma \ref{lem:APS} and let $\tilde{C} := (1-\eta)c(v)C  \in \bB(L^2(M_\infty , S_{E, \cV}))$ as in Subsection \ref{section:2.2}. Then, there is an odd self-adjoint smoothing operator $A \in \Psi^{-\infty}_c(M_\infty , S_{E, \cV})$ with compact support such that $(D+\tilde{C}+ A)^2 \geq \lambda \cdot 1$ for some $\lambda >0$. 

Let $\chi$ be as in (\ref{form:Ft}) and let $\chi_s(t):=\chi((1-s)^{-1}t)$. Set \begin{align*}
\tilde{F}_{\cV, C}(s, \lambda ):=\chi _s(D_\cV + s\tilde{C} + \lambda A) \in \bB(L^2(M_\infty , S_{E, \cV})).
\end{align*}
Then,  $\tilde{F}_\cV(0, 0) = F_{\cV}$ and $\| \tilde{F}_\cV(s, 1)^2 - 1\| \to 0 $ as $s \to 1$ hold.

\begin{lem}\label{lem:diffineq}
We have 
\[ \| [\rho^{(1-s)^2}_\tau, F_{\cV , C}(s, \lambda ) ] \| \to 0 \text{ as $s \to 1$} \]
uniformly on $(\tau, \lambda) \in [0,1] \times [0,1]$.
\end{lem}
\begin{proof}
By definition of $\tilde{C}$, it commutes with $\rho_\tau^{(1-s)^2}$. Hence the commutator
\[  [D_\cV + s\tilde{C} + \lambda A , \rho _\tau^{(1-s)^2}] = c(d\rho_\tau ^{(1-s)^2}) + s[ \kappa A , \rho_\tau^{(1-s)^2}] \]
(where $c(d\rho_\tau^{(1-s)^2})$ is the Clifford multiplication of the $1$-form $d(\rho_\tau^{(1-s)^2})$ on the spinor bundle $S$) has the norm less than $\| d(\rho^{(1-s)^2}_\tau) \| =(1-s)^2$ for sufficiently large $s \in [0,1)$ such that $A (\rho_\tau ^{(1-s)^2} +1)=(\rho_\tau^{(1-s)^2} +1)A =0$. By the formula $[A, e^B]= \int_0^1 e^{\tau B}[A,B]e^{(1-\tau )B}d\tau $, we get
\begin{align*}
&\| [e^{iu(1-s)^{-1}(D_\cV + s\tilde{C} + \lambda A)} , \rho^{(1-s)^2}_\tau] \| \\
 \leq& (1-s)^{-1}\int _0^1 \| [iu(D_\cV + s\tilde{C} + \lambda A) , \rho_\tau ^{(1-s)^2}] \| d\tau \\
 \leq& |u| (1-s).
\end{align*}
Therefore we obtain that
\begin{align*}
\| [\rho^{(1-s)^2}_\tau, F_{\cV , C}(s, \lambda) ]\| & \leq  \frac{1}{2\pi}\int_{u \in \bR} \hat{\chi}(u)\| [e^{iu(1-s)^{-1}(D_\cV + s\tilde{C} + \lambda A)} , \rho^{(1-s)^2}_\tau] \| du \\
& \leq (1-s)\frac{1}{2\pi}\int_{t \in \bR } |\hat{\chi}(u)||u| du \to 0 \text{ as $s \to 1$}. \qedhere
\end{align*}
\end{proof}

For $s \in [0,1)$ and $\kappa \in [0,1]$, let
\begin{align*}
\tilde{T}^1_\kappa (s)&:= 
\begin{pmatrix}
\rho_s^{(1-s\kappa)^2 } &  \sigma_s^{(1-s\kappa)^2} \tilde{F} _\cV(\kappa s,\kappa s)^1 \sigma_s ^{(1-s\kappa)^2} \\
\sigma _s ^{(1-s\kappa)^2} \tilde{F}_{\cV }(\kappa s, \kappa s)^0 \sigma _s ^{(1-s\kappa)^2 } & -\rho _s ^{(1-s\kappa)^2} 
\end{pmatrix}_{\textstyle ,} \\
\tilde{T}^2_\kappa (s)&:= 
\begin{pmatrix}
\rho_{s\kappa }^{(1-s)^2} &  \sigma_ {s\kappa } ^{(1-s)^2} \tilde{F}_\cV(s,s)^1 \sigma_ {s\kappa }^{(1-s)^2} \\
\sigma ^{(1-s)^2} _ {s\kappa }  \tilde{F}_{\cV }(s,s)^0 \sigma ^{(1-s)^2} _ {s\kappa } & -\rho ^{(1-s)^2} _ {s\kappa } 
\end{pmatrix}_{\textstyle ,} \\
\tilde{T}^3_\kappa (s)&:=
\begin{pmatrix}
\rho_{0}^{(1-s)^2} &  \sigma_ {0} ^{(1-s)^2} \tilde{F}_\cV(s, s\kappa  )^1 \sigma_ {0}^{(1-s)^2} \\
\sigma ^{(1-s)^2} _ {0}  \tilde{F}_{\cV }(s, s\kappa  )^0 \sigma ^{(1-s)^2} _ {0} & -\rho ^{(1-s)^2} _ {0} 
\end{pmatrix} _{\textstyle .}
\end{align*}
Then, they satisfy 
\begin{enumerate}
\item $\tilde{T}^i_\kappa (0)=T_\cV(0)$, 
\item $(\tilde{T}^i_\kappa (s))^2-1 \in \bK(L^2(M_\infty , S_{E, \cV}))$ for any $s\in [0,1)$ and
\item $\| (\tilde{T}^i_\kappa (s))^2-1 \| \to 0$ as $s \to 1$,
\end{enumerate}
for any $\kappa  \in [0,1]$ and $i=1,2,3$. Indeed, (3) for $\tilde{T}^2_\kappa$ and $\tilde{T}^3_\kappa$ follows from Lemma \ref{lem:diffineq}, $\| \tilde{F}_{\cV}(s, 1)^2 -1 \| \to 0$ and 
\[ \| \sigma^{(1-s)^2}_0 (F(s, \lambda)^0 - F(s,1)^0) \sigma^{(1-s)^2}_0 \| \to 0 \text{ as $s \to 1$ }\]
for any $\lambda \in [0,1]$, which follows from $F(s, \lambda) - F(s,1) \in \bK(L^2(M_\infty , S_{E , \cV}))$.

Now, (1), (2) and (3) means that $(T_\cW, \tilde{T}_\kappa ^i(s)) \in \cM(C\phi)$ satisfies 
\[ (T_\cW, \tilde{T}_\kappa ^i(s))^2 -1 \in C\phi\]
for any $\kappa  \in [0,1]$ and $i=1,2,3$. Since $\tilde{T}^1_1 = \tilde{T}^2_1$ and $\tilde{T}^2_0 = \tilde{T}^3_1$, we get 
\[ \mu^{\MF}_0([M, \id, E])= [(T_\cW , \tilde{T}_0^1(s))] = [(T_\cW, \tilde{T}^3_0(s))] \in \KR_0 (\cQ(L^2(M_\infty, S_{E, \cE} ))). \]  
Note that 
\begin{align} \tilde{T}_0^3(s)= \begin{pmatrix}\rho_0^{(1-s)^2} & \sigma_0^{(1-s)^2} \chi_s(D_{\cV}+sC) ^* \sigma_0^{(1-s)^2} \\ \sigma_0^{(1-s)^2} \chi^s(D_{\cV}+sC) \sigma_0^{(1-s)^2} & -\rho_0^{(1-s)^2} \end{pmatrix}_{\textstyle .} \label{form:operator1} \end{align}

\subsubsection*{Step 3}
Next we use unbounded $\KK$-theory \cite{MR715325} in order to give another presentation of the element $c(\bD_\cW , C) = \mu^\DG_0([M,\id, E])$.
Let $S_{E, \cX}$ denote the Hilbert $C\phi$-module bundle $S_E \otimes \cX$ on $N$. Consider the regular self-adjoint operators
\begin{align*}
\bD_{\cX, C} &:= (\bD_\cW, (1-s)^{-1}(\bD_{\cV}+sC)), \\
D_{\cX , C} &:= (D_\cW , (1-s)^{-1}(D_{\cV , \infty} + sc(v)C) ) \\
&= c(v)\Big( (1-s)^{-1}\frac{d}{dr} + \bD_{\cX ,C} \Big),
\end{align*}
acting on $L^2(N , S_{E, \cX})$ and $L^2(N \times \bR , S_{E , \cX})$ respectively. They have compact resolvent and $[L^2(N,S_{E, \cX}), 1 , \bD_{\cX, C}]$ determines an real self-adjoint unbounded Kasparov $\bR$-$C\phi$ bimodule representing $c(\bD_\cW, C)$. Let $\alpha$ denote the Bott generator $[L^2(\bR), m , i\frac{d}{dt}] \in \KKR _1 (C_0(\bR), \bR)$ (here $m$ denote the multiplication representation of $C_0(\bR)$). By \cite[Th\'{e}or\`{e}me 3.2]{MR715325}, we have
\begin{align*} 
\alpha \otimes c(\bD_\cW , C) &= \bigg[ L^2(N \times \bR , S_{E, \cX}), m, \begin{pmatrix} 0 & - \frac{d}{dr} + \bD_{\cX, C} \\ \frac{d}{dr} + \bD_{\cX, C} & 0\end{pmatrix}
\bigg]\\
&= \bigg[ L^2(N \times \bR , S_{E , \cX}), m_{1-s}, \begin{pmatrix} 0 & - \frac{d}{dr} + \bD_{\cX, C} \\ \frac{d}{dr} + \bD_{\cX, C} & 0\end{pmatrix}
\bigg] \\
&= \bigg[ L^2(N \times \bR , S_{E , \cX}), m_{ (1-s)^2 }, \begin{pmatrix} 0 & D_{\cX, C}^* \\ D_{\cX, C} & 0\end{pmatrix}
\bigg]_{\textstyle , }
\end{align*}
where $m_s(f)$ denote multiplication by $f^s$. Here, the second equality is given by a continuous path $ \{ m_{1-\kappa s} \}_{\kappa \in [0,1]}$ of $\ast$-homomorphisms and the third equality is given by the adjoint with respect to the unitary $U_s(f)(x):=sf((1-s)^{-1}x)$ on $L^2(\bR)$.

The corresponding bounded Kasparov bimodule is
\[  \bigg[ L^2(N \times \bR , S_{E, \cX} ), m_{(1-s)^2} ,  \begin{pmatrix} 0 & \chi(D_{\cX , C}) \\ \chi(D_{\cX, C}) & 0 \end{pmatrix}
 \bigg]_{\textstyle .} \]
Recall that $\chi(D_{\cX,C}) = (\chi(D_\cW), \chi_s(D_{\cV,\infty}+sC))$. By Lemma \ref{lem:Kas}, we get
\begin{align*}
c(\bD_\cW, C)&=\beta \otimes_{C_0(\bR)} \alpha \otimes c(\bD_\cW, C)\\
&= [ L^2(N \times \bR , S_{E, \cX} ) , 1 , (T_\cW, T_\cV'(s)) ], 
\end{align*}
where \footnotesize
\begin{align}  T_\cV'(s)=  \begin{pmatrix} \rho_0^{(1-s)^2} & \sigma_0^{(1-s)^2} \chi_s(D_{\cV , \infty} +sc(v)C ) ^* \sigma_0^{(1-s)^2} \\ \sigma_0^{(1-s)^2} \chi_s(D_{\cV , \infty} +sc(v)C) \sigma_0^{(1-s)^2} & -\rho_0^{(1-s)^2} \end{pmatrix}_{\textstyle .} \label{form:operator2} \end{align} \normalsize

\subsubsection*{Step 4}
Finally we show the representative of in Step 2 is canonically identified with the representative of $c(\bD_\cW , C)$ given in Step 3.
More precisely, let $V_1 \colon  L^2(N_\infty , S_{E, \cX}) \to L^2(N\times \bR , S_{E, \cX})$ and $V_2 \colon L^2(N_\infty , S_{E, \cX}) \to L^2(M_\infty, S_{E, \cE})$ denote the canonical isometries. Then $\Ad (V_1)$ and $\Ad (V_2)$ induce isomorphisms of $\K$-theory as
\begin{align*} 
\begin{split}
\xymatrix@R=0.5em{
&\KR_0(\cQ( L^2(N\times \bR , S_{E, \cX}))) \\
\KR_0(\cQ (L^2(N_\infty  , S_{E, \cX}))) \ar[ru]^{\Ad (V_1)} \ar[rd]_{\Ad (V_2)} & \\
& \KR_0 (\cQ (L^2(M_\infty, S_{E, \cE}))). 
}
\end{split}
\label{form:identK}\end{align*}
\begin{lem}
We have 
\begin{enumerate}
\item $(T_\cW , T_\cV'(s))  = V_1V_1^* (T_\cW , T_\cV'(s))  V_1V_1^* + (1-V_1V_1^*) \big( \begin{smallmatrix} 1 & 0 \\ 0 & -1 \end{smallmatrix} \big)$, 
\item $ ( T_\cW , \tilde{T}_0^3(s)) = V_2V_2^* ( T_\cW , \tilde{T}_0^3(s))V_2V_2^* + (1-V_2V_2^*)\big( \begin{smallmatrix} 1 & 0 \\ 0 & -1 \end{smallmatrix} \big)$, and
\item $V^*_1(T_\cW , T_\cV'(s)) V_1=  V_2^* ( T_\cW , \tilde{T}_0^3(s))V_2$. 
\end{enumerate}
\end{lem}
This shows 
\[ [(T_\cW, T_\cV')]=[(T_\cW , \tilde{T}_0^3)]\]
which completes the proof of $\mu_0^\MF ([M,\id, E]) = \mu_0^\DG ([M,\id , E])$ by Step 2 and 3.
\begin{proof}
The claim (1) and (2) are obvious from the fact that support of $\rho_0^{(1-s)^2}$ lies in $N \times [0,\infty)$. Hereafter we show (3).

Let us compare (\ref{form:operator1}) and (\ref{form:operator2}).
Since the function $\sigma_0^{(1-s)^2}$ is supported on $N \times [(1-s)^{-1}, \infty)$, it is enough to see that
\begin{align*}  \chi_s(D_{\cV} + s\tilde{C})V_1 \xi = \chi_s(D_{\cV ,\infty} + sc(v)C)V_2 \xi  \end{align*}
holds for a smooth section $\xi \in C_c^\infty (N_\infty , S_{E , \cX})$ supported on $N \times [(1-s)^{-1} , \infty)$. 

Let $S_{E,\cX}=S_{E, \cX}^{+} \oplus S_{E, \cX}^{-}$ denote the eigenspace of $c(v)$ with the eigenvalue $\pm i$.
We decompose $\xi \in C_c^\infty (N \times \bR , S_{E , \cX})$ as $\xi=\xi^+ + \xi^-$ along this bundle decomposition and define $\xi_t \in C_c^\infty (N \times \bR , S_{E , \cX})$ as
\[ \xi_t(y,r):=e^{-it(\bD_{\cV} + sC)} \xi^+ (y ,r-t ) + e^{it(\bD_{\cV} + sC)} \xi^- (y ,r+t ). \]
Then it satisfies $\supp (\xi_t) \subset N \times [(1-s)^{-1}-t , \infty)$ and 
\begin{align*}
\frac{d}{dt}\Big|_{t=t_0} V_1 \xi_t &= ic(v) \Big( \frac{d}{dr} + \bD_{\cV} + sC \Big) \xi_{t_0} = i(D_{\cV} + s \tilde{C}) V_1  \xi_{t_0}, \\
\frac{d}{dt}\Big|_{t=t_0}  V_2 \xi_t &= i c(v) \Big(\frac{d}{dr} + \bD_{\cV} + sC \Big) \xi_{t_0} =  i(D_{\cV, \infty}+ sc(v)C) V_2 \xi_{t_0}. \\
\end{align*}
That is, $e^{it (D_{\cV} + s\tilde{C})}V_1 \xi = V_1 \xi_t$ and $e^{it (D_{\cV , \infty}+ sc(v)C)} V_2\xi = V_2 \xi_t$ holds. This completes the proof by the Fourier transform presentation (\ref{form:Ft}).
\end{proof}

\section{Partitioning and the K-theoretic van Kampen theorem}\label{section:9}
In this section, we relate the relative higher index with the higher index of the amalgamated free product group. 
Let $\phi_1 \colon \Lambda \to \Gamma _1 $ and $\phi_2 \colon \Lambda \to \Gamma _2$ be two homomorphisms of groups. Here we assume that $\phi_1$ and $\phi_2$ are injectve. Let $f_1 \colon (X^1, Y) \to (B\Gamma_1, B\Lambda )$ and $f_2 \colon (X^2,Y) \to (B\Gamma _2 , B\Lambda)$ be continuous maps such that $f_1|_Y=f_2|_Y$. Set $\Gamma := \Gamma _1 \ast_\Lambda \Gamma _2$ and $\hat{X}:= X_2^1 \sqcup X_2^2 / \sim$, where the relation $\sim$ identifies $(y,1+r) \in Y_2' \subset X^1_2$ with $(y, 2-r) \in Y_2' \subset X^2_2$ for $r \in [0,1]$. Then $f_1$ and $f_2$ gives rise to a continuous map 
\[ f \colon X \to B\Gamma_1 \sqcup_{B\Lambda}B\Gamma_2 = B\Gamma  \] 
and hence we have assembly maps
\begin{align*}
\mu _*^{\Gamma _i, \Lambda} & \colon \K_*(X^i ,Y) \to \K_*(C^*(\Gamma_i , \Lambda)),
 \\
\mu_*^\Gamma & \colon \K_*(\hat{X}) \to \K_*(C^*(\Gamma )). 
\end{align*}
Now, we have the Mayer--Vietoris exact sequence 
\[\cdots \to \K_*(Y) \to \K_*(\hat{X}) \to \K_*(X^1,Y) \oplus \K_*(X^2, Y) \to \K _{*-1}(Y) \to \cdots .\]

There is a corresponding exact sequence of group C*-algebras. This follows from the KK-equivalence
\begin{align*}
C^*(\Gamma _1 \ast_\Lambda \Gamma _2) \sim_{\KK}&  SC(\phi _1 \oplus \phi_2) \\
=& C^*(\Gamma _1 ,\Lambda) \oplus _{SC^*(\Lambda)}C^*(\Gamma _2 , \Lambda ), 
\end{align*}
which is first proved implicitly by Pimsner \cite{MR860685}. It is pointed out in \cite{MR1426836} that this $\KK$-equivalence is given by the inclusion $C(\phi_1 \oplus \phi_2 ) \to C^*(\Gamma)(-1,1)$ mapping $(a_1,a_2) \in C(\phi_1 \oplus \phi_2)$ to
\[ a(s) := \left\{ \begin{array}{ll}a_1(s) & s \in [0,1), \\ a_2(-s) & s \in (-1,0]. \end{array} \right. \]
Then, we have $\ast$-homomorphisms
\[ \psi_i \colon S^{0,1}C(\phi _1 \oplus \phi_2 ) \to S^{0,1}C\phi _i , \ \ \psi_i(f_1,f_2)=f_i, \]
for $i=1,2$. 

\begin{prp}\label{prp:cut}
The diagram
\[\xymatrix@C=1em{
\cdots \ar[r] &  \KO_*(Y) \ar[r]^{i_*} \ar[d]^{\mu^\Lambda_*} &  \KO_*(\hat{X}) \ar[d]^{\mu^\Gamma_*} \ar[rr] ^{(j_1)_* \oplus -(j_2)_* \ \ \  } && {\begin{array}{c} \K_*(X^1,Y)\\ \oplus\\ \KO_*(X^2, Y)\end{array}} \ar[rr]^{\partial_1 \oplus \partial_2 } \ar[d]^{\mu^{\Gamma _1,\Lambda }_* \oplus \mu^{\Gamma _2, \Lambda}_*} && \K O_{*-1}(Y) \ar[r] \ar[d]^{\mu_{*-1}^\Lambda} & \cdots \\
\cdots \ar[r] & \KR_*(C^*(\Lambda)) \ar[r]^{\phi_*} & \KR_*(C^*(\Gamma)) \ar[rr]^{(\psi_1)_* \oplus (\psi_2)_* \ \ } && {\begin{array}{c}\KR_*(C^*(\Gamma_1, \Lambda)) \\ \oplus \\ \KR_*(C^*(\Gamma _2, \Lambda)) \end{array}} \ar[rr] ^{(\theta _1)_* \oplus (\theta_2)_*} && \KR_{*-1}(C^*(\Lambda) ) \ar[r] & \cdots 
}\]
commutes.
\end{prp}
\begin{proof}
It suffices to show that the middle square commutes. Let $\sE^1_2$ and $\sE^2_2$ denote the Hilbert $C_0((X_2^i)^\circ) \otimes C\phi_i$-module as in (\ref{form:E}) for $X^1$ and $X^2$ respectively and set
\[ \hat{\sE}:= \{ (\xi_1, \xi_2) \in \sE^1_2 \oplus \sE_2^2 \mid \xi_1(y,1+r,0) =\xi_2(y,2-r,0) \text{ for any $r \in [0,1]$} \}, \]
which is canonically regarded as a Hilbert $C(\hat{X}) \otimes C(\phi_1 \oplus \phi_2)$-module. For $i=1,2$, we write $r_i$ for continuous functions on $\hat{X}$ extending $r$ on $X_2^i$ as $r_i \equiv 2$ on $\hat{X} \setminus X_2^i$.  We define $\hat{\rho} \in C(\hat{X})(-1,1)$ as
\[ \hat{\rho} (x,  s) = \left\{ \begin{array}{ll}\rho(r_1(x) ,s) & s \in [0,1) \\ - \rho (r_2(x),-s) & s \in (-1,0]. \end{array} \right. \]
\begin{figure}[t]
\begin{tikzpicture}[scale=0.4]
\shade [top color = black, bottom color = black, middle color = white, shading angle = 135] (-3,-4) rectangle (9,4);
\shade [top color = black!35!white, bottom color = black!35!white, middle color = white] (-3,0) rectangle (0,4);
\shade [top color = black!35!white, bottom color = black!35!white, middle color = white] (6,0) rectangle (9,-4);
\fill [color = black!35!white] (6,0) -- (6,4) -- (0,4) -- (6,0);
\fill [color = black!35!white] (6,-4) -- (0,-4) -- (0,0) -- (6,-4);
\fill [color = white] (-3.1,-4.1) rectangle (0,0);
\fill [color = white] (9.1,4.1) rectangle (6,0);
\draw [->] (-4.5,-4.3) -- (-4.5,4.5);
\draw [->] (9,-6) -- (-2,-6);
\draw [->] (-3,6) -- (8,6);
\draw [thick,dashed] (-3,4) -- (6,4);
\draw [thick, dashed] (6,4) -- (6,0);
\draw [thick, dashed] (6,0) -- (9,0);
\draw [thick,dashed] (0,0) -- (-3,0);
\draw [thick,dashed] (0,0) -- (0,-4);
\draw [thick,dashed] (0,-4) -- (9,-4);
\node at (-1.4,-5.6) {$r_2$};
\node at (7.4,5.6) {$r_1$};
\node at (-3.9,4.3) {$s$};
\node at (-1.5,5) {$X_1$};
\node at (3,5) {$Y_2'$};
\node at (7.5,-5) {$X_1$};
\node at (3,-5) {$Y_2'$};
\fill (6,-6) coordinate (O) circle[radius=2pt] node[below=2pt]  {$1$};
\fill (0,-6) coordinate (O) circle[radius=2pt] node[below=2pt]  {$2$};
\fill (6,6) coordinate (O) circle[radius=2pt] node[above=2pt]  {$2$};
\fill (0,6) coordinate (O) circle[radius=2pt] node[above=2pt]  {$1$};
\fill (-4.5,4) coordinate (O) circle[radius=1pt] node[left=2pt]  {$1$};
\fill (-4.5,0) coordinate (O) circle[radius=1pt] node[left=2pt]  {$0$};
\fill (-4.5,-4) coordinate (O) circle[radius=1pt] node[left=2pt]  {$-1$};
\end{tikzpicture}
\caption{The shading shows the value of $| \tilde{\rho}(r,s) |$.}
\end{figure}
Then the triplet $(\hat{\sE}, 1 , \hat{\rho})$ determines a real self-adjoint Kasparov bimodule 
\[\ell_{\Gamma_1, \Lambda , \Gamma _2} := [ \hat{\sE},1, \hat\rho ] \in \KKR_1(\bC, C(X) \otimes C(\phi_1 \oplus \phi_2)).\]

By definition we have $\ell_{\Gamma _1, \Lambda , \Gamma _2} \otimes [\psi_1] = \ell_{\Gamma , \Lambda }$ and $\ell_{\Gamma _1, \Lambda , \Gamma _2 } \otimes [\psi_2] = -\beta_{\Gamma _2, \Lambda }$. Moreover
\[\ell_{\Gamma_1, \Lambda , \Gamma_2 } \otimes [\iota] = [C_0(\hat{X} , \hat{\cV})(-1,1), 1 , 2s-1 ]= \ell_\Gamma \otimes \beta, \]
where $\hat{\cV}$ is the Mishchenko line bundle on $\hat{X}$ with respect to the universal covering (with the fiber $\pi_1(\hat{X}) \cong \Gamma$).
That is, the Kasparov product with $\ell_{\Gamma _1 ,\Lambda , \Gamma _2}$ is equal to the higher index $\mu^\Gamma_*$. Now the proof is completed.
\end{proof}

\begin{thm}\label{thm:cut}
Let $M^1$ and $M^2$ be two spin manifolds with the same boundary $N$ such that $\Lambda := \pi_1(N) \to \pi_1(M^i) =:\Gamma _i$ are injective. Let $M:= M^1 \sqcup _N M^2$ and $\Gamma:=\Gamma _1 \ast _\Lambda \Gamma _2$. Then, 
\[(\psi_i)_* \circ \mu ^{\Gamma }_*([M])=\mu_*^{\Gamma _i, \Lambda}([M^i, N]). \]
In particular, the non-vanishing of  $\mu _*^{\Gamma _i, \Lambda }([M^i, N])$ for one of $i=1,2$ implies $\mu_*^\Gamma ([M]) \neq 0$.
\end{thm}
\begin{proof}
It follows from Proposition \ref{prp:cut} and $(j_i)_*([M])=[M^i,N]$.
\end{proof}

\begin{cor}\label{cor:hyper}
Let $M$ be a closed spin manifold partitioned to $M=M^1 \sqcup_N M^2$ by an oriented hypersurface $N$. Let $\Lambda :=\pi_1(N)$ and $\Gamma:=\pi_1(M)$. We assume that $\Lambda \to \Gamma $ is injective. Then $\mu_{*-1}^\Lambda([N]) \neq 0$ implies the nonvanishing of $\mu_{*}^\Gamma([M])$. 
\end{cor}
\begin{proof}
It follows from Theorem \ref{thm:cut} and the commutative diagram (\ref{cond:long}). Note that $\Lambda \to \Gamma$ is injective if and only if both $\Lambda \to \pi_1(M^1)$ and $\Lambda \to \pi_1(M^2)$ are injective.
\end{proof}
This is analogous to the partitioned manifold index theorem proved in a completely different way to existing approaches such as \cite{MR1817560,MR2670972}. In particular, we can apply this corollary to get the non-vanishing of higher indices for manifolds partitioned by an enlargeable manifold.

Theorem \ref{thm:cut} is also applied to invariance of non-vanishing of the higher index under cutting-and-pasting. This is different from the invariance of higher indices under cutting-and-pasting of Galois coverings studied in \cite{MR1905836} in the sense that the fundamental group can be changed.
\begin{cor}
Let $M$ be a closed spin manifold partitioned by an oriented hypersurface $N$. Assume that $\mu_*^{\Gamma}([M])$ is not in $\Im \phi_* \subset \K_*(C^*(\Gamma))$. Let $\psi$ be a diffeomorphism of $N$ preserving the spin structure and let $\hat{M}:=M^1 \sqcup_{\psi} M^2$ be the spin manifold obtained by cutting-and-pasting. Then $\mu_*^{\hat{\Gamma}}([\hat{M}]) \in \KR_*(C^*(\hat{\Gamma}))$ does not vanish, where $\hat{\Gamma}:=\pi_1(\hat{M})$. 
\end{cor}

Finally we go back to the index theory of invertible doubles. As is pointed out at the beginning of this section, a spin manifold with boundary has finite relative K-area if and only if its invertible double has finite K-area. The corresponding result in higher index is the following.
\begin{cor}
Let $M$ be a compact spin manifold with the boundary $N$ and let $\Gamma :=\pi_1(M)$ and $\Lambda :=\pi_1(N)$. Assume that $\phi \colon \Lambda \to \Gamma$ is injective. Then $\mu^{\Gamma, \Lambda}_*([M,N])=0$ if and only if $\mu_*^{\Gamma \ast _\Lambda \Gamma }([\hat{M}])=0$.
\end{cor}
\begin{proof}
Let $\chi$ denote the $\ast$-homomorphism $C\phi \to C(\phi \oplus \phi)$ given by $\chi(a,b_s)  = (a,b_s \oplus b_s)$. Then, $\psi_i \circ \chi _* =\id_{C\phi}$ and hence 
\[ \KR_*(C^*(\Gamma \ast_\Lambda \Gamma )) \cong \KR_*(C^*(\Gamma , \Lambda)) \oplus \KR_*(C^*\Gamma ). \] 
Through this isomorphism, $\mu^{\Gamma \ast_\Lambda \Gamma}([\hat{M}])$ is identified with $\mu^{\Gamma , \Lambda}_* ([M,N]) \oplus \mu_* ^{\Gamma}([\hat{M}])$. Since the higher index of the invertible double $\mu ^{\Gamma}_*([\hat{M}])$ vanishes (\cite[Theorem 5.1]{MR3122162}), we complete the proof.
\end{proof}
\begin{cor}
Let $M$ be a compact spin manifold with the boundary $N$. Assume that $\phi \colon \Lambda \to \Gamma$ is injective. If its invertible double $\hat{M}$ is enlargeable, then $\mu^{\Gamma ,\Lambda}_*([M,N])$ does not vanish. In particular, for any spin-oriented open embedding of $M^\circ$ to a closed spin manifold $\bar{M}$ such that $\Lambda \to \pi_1(\bar{M})$ is injective, $\bar{M}$ does not have any metric with positive scalar curvature.
\end{cor}

We close this section by introducing an alternative proof of Theorem \ref{thm:cut} using the Deeley--Goffeng description $\mu_*^{\DG}$. Let  $\cV_1$ and $\cV_2$ denote the Mishchenko line bundle on $M^1$ and $M^2$ respectively, let $D_i$ denote the Dirac operator on $M_{\infty}^i$ and let $C_i$ denote the smoothing operator as in Subsection \ref{section:2.2}.
\begin{prp}
Let $M^i$, $N$, $M$, $\Gamma _i$, $\Lambda$ and $\Gamma$ be as in Theorem \ref{thm:cut}. Then 
\[ \mu^\Gamma _* ([M])= (\psi_1)_* \ind _b (D_{\cV_1}, C_1) - (\psi_2)_* \ind_b(D_{\cV_2},C_2)+c(\bD_{\cW} , C_1, C_2). \]
Here, $c(\bD_\cW, C_1,C_2):=[(U, V_s^1 \oplus V_s^2)] \in \KR_1(C(\phi_1 \oplus \phi _2))$, where $(U , V_s^1)$ and $(U,V_s^2)$ are as in Subsection \ref{section:2.2} for $(M^1,N)$ and $(M^2,N)$ respectively.
\end{prp}
\begin{proof}
Let $\cP_1 \in \bB(L^2(N, S_{E, \cV_1}))$ and $\cP_2 \in \bB(L^2(N, S_{E, \cV_1}))$ denote the positive spectral subspaces of $\bD_{\cV_1}$ and $\bD_{\cV_2}$ respectively. Then, the image of $c(\bD_\cW, C_1 , C_2)$ in $SC^*\Gamma$ determines an element of $\K_1(SC^*\Gamma)$ with the image $[\cP_1 - \cP_2]$. Now the equality follows from $\ind_{\mathrm{APS}}(D_\cV , \cP)=\ind_b(D_\bV , C)$ (\cite[Theorem 5]{MR1979016}) and the gluing formula 
\[ \ind (D_{\hat{\cV}})= \ind_{\mathrm{APS}}(D_{\cV_1} , \cP_1) - \ind_{\mathrm{APS}}(D_{\cV_2} , C_2) + [\cP_1 - \cP_2]   \]
shown in \cite[Theorem 8]{MR1979016}.
\end{proof}

\section{Rational injectivity and related results}\label{section:4}
In this section, we investigate the rational surjectivity of the universal envelope
\[ \beta _{\Gamma, \Lambda } \colon \K^0(C^*(\Gamma, \Lambda )) \to \K^0 (B\Gamma, B \Lambda ) \]
of the dual relative higher index maps in Definition \ref{defn:dBC}, which will be related with almost flat vector bundles in the second paper \cite{Kubota2}. This follows from the non-degeneracy of the pairing of the relative $\K$-group $\K^*(B\Gamma, B\Lambda )$ and $\K$-homology group of $C^*(\Gamma, \Lambda)$. To this end, we also study the the rational injectivity of the relative assembly map.  

The key ingredient of the proof is the Dirac--dual Dirac method~\cite{MR918241}. 
It is convenient to work in the categorical framework of the $\KK$-theory and the Baum--Connes assembly map introduced in \cite{MR2193334}. 
Let $\Csep^\Gamma $ denote the category of separable $\Gamma$-C*-algebras and equivariant $\ast$-homomorphisms and let $\Kas^\Gamma$ denote the equivariant Kasparov category (the category whose objects are separable $\Gamma$-C*-algebras, morphisms are equivariant $\KK$-group and the composition is given by the Kasparov product). 

The functor $\KK^\Gamma \colon \Csep^\Gamma \to \Kas^\Gamma$ has the universal property \cite{MR899916,MR1803228} that any stable, homotopy invariant and split exact additive functor $F$ from $\Csep^\Gamma$ to an additive category factors through $\KK^\Gamma$. That is, an element $\xi \in \KK^\Gamma(A,B)$ induces a morphism $F(\xi) \colon F(A) \to F(B)$. For example, the composition of the maximal crossed product functor  
\[ j_\Gamma \colon \Csep^\Gamma \to \Csep, \ \ j_{\Gamma}(A) := A \rtimes \Gamma \] 
with $\KK$ factors through an additive functor $j_\Gamma \colon \Kas^\Gamma \to \Kas$. Indeed, this is the same thing as the descent functor in \cite[Theorem 3.11]{MR918241}. 

A homomorphism $\phi \colon \Lambda \to \Gamma $ induces the pull-back functor 
\[ \phi^* \colon \Csep^\Gamma \to \Csep ^\Lambda. \]
For a $\Gamma$-C*-algebra $A$, the $\ast$-homomorphism
\[ \Phi_A\Big( \sum _{\gamma} a_\gamma u_{\gamma} \Big) = \sum_\gamma a_{\gamma} u_{\phi(\gamma)}\]
between algebraic crossed products is extended to a well-defined $\ast$-homomorphism $\Phi _A \colon A \rtimes \Lambda \to A \rtimes \Gamma$. 
Note that $\Phi_\bC $ is the same thing as $\phi \colon C^*\Lambda \to C^*\Gamma$. The family $\{ \Phi_A\}_{A \in \Csep^\Gamma }$ determines a natural transform from $j_\Lambda \circ \phi ^*$ to $j_\Gamma$. Let $j_\phi \colon \Csep^\Gamma \to \Csep$ denote its mapping cone functor. More explicitly, let $j_\phi(A):= C\Phi_A$ and $j_\phi(\varphi) \colon C\Phi_A \to C\Phi_B$ given by
\[ j_\phi(\varphi)(a,b_s)=(((j_\Lambda \circ \phi^*)(\varphi))(a), (j_\Gamma (\varphi))(b_s)). \]
By the universality, we obtain the functor 
\[j_\phi \colon \Kas ^\Gamma \to \Kas  \]
between the corresponding Kasparov categories.

We say that a countable discrete group $G$ has the $\gamma$-element if there is a proper $G$-C*-algebra $A_G$ in the sense of \cite[Definition 1.6]{MR1836047} (or, more generally, $A_G$ in the subcategory $\langle \mathcal{CI} \rangle$ of $\mathfrak{KK}^G $ in the sense of \cite[Definition 4.1]{MR2193334}) and 
\[ \mathsf{D}_G \in \KK^G (A_G, \bC ), \ \  \eta_G \in \KK^G (\bC, A_G ), \]
called the Dirac and dual Dirac elements respectively, such that
\[ \mathsf{D} _G \otimes \eta _G = \id _{A_G}, \ \ \Res _G^K (\eta_G \otimes _{A_G}\mathsf{D}_G)=\id _\bC, \]
for any finite subgroup $K$ of $G$. Here $\gamma _G := \eta _G \otimes _{A_G} \mathsf{D}_G \in \KK^G(\bC, \bC )$ is called the $\gamma $-element of $G$.

Let $\phi \colon \Lambda \to \Gamma$ be a homomorphism between countable discrete groups. In this section we work under the following assumptions:
\begin{itemize}
\item[\eqnum \label{cond:BC1}] The group $\Gamma$ has the $\gamma$-element. \setcounter{copy}{\value{equation}} 
\item[\eqnum \label{cond:BC2}] For any finite subgroup $K \subset \Gamma$, the subgroup $\phi^{-1}(K) \leq \Lambda$ satisfies $\gamma=1$. 
\item[\eqnum \label{cond:BC3}] The subgroup $\ker \phi$ is torsion-free. 
\end{itemize}
For example, the condition (\ref{cond:BC1}) is satisfied if $\Gamma$ is coarsely embeddable into a separable Hilbert space \cite{MR1905840} and the condition (\ref{cond:BC2}) is satisfied if $\ker \phi$ has the Haagerup property \cite{MR1821144}. 

By (\ref{cond:BC1}), we get a homomorphism
\[  j_{\Gamma }(\eta_\Gamma ) \circ \mu^\Gamma_* \colon \K_*^\Gamma (\underline{E}\Gamma ) \to \K_*(A_\Gamma \rtimes \Gamma).  \]
This is nothing but the Baum--Connes assembly map with coefficient $A_\Gamma$ and hence is an isomorphism by \cite[Corollary 3.7]{MR1966758}. 
Moreover, the composition
\[ (j_{\Lambda}\circ \phi^*)(\eta_\Gamma ) \circ \mu^\Lambda_* \colon \K _*^\Lambda (\underline{E}\Lambda ) \to \K (A_\Gamma \rtimes \Lambda )\]
is actually an isomorphism by the permanence property of the Baum--Connes isomorphism \cite[Corollary 3.4]{MR1836047}. 

By the universality of $\underline{E}\Gamma$, there is a $\Gamma$-equivariant map $f_\Gamma \colon E\Gamma \to \underline{E}\Gamma$. Under the isomorphism $\K_*(B\Gamma) \cong \K_*^\Gamma (E\Gamma)$, the map $\K_*(B\Lambda) \to \K_*(B\Gamma)$ is identified with the composition of the isomorphism $\K_*^\Lambda(E\Lambda) \cong \K_*^\Gamma (E\Lambda \times _{\Lambda} \Gamma)$ with
\[E\phi_* \colon \K_*^\Gamma (E\Lambda \times _{\Lambda} \Gamma) \to \K_*^\Gamma (E\Gamma ), \]
where $E\phi \colon E\Lambda \times _{\Lambda} \Gamma \to E\Gamma$ is the $\Gamma$-map induced from the universality of $E\Gamma$. Similarly, let $\underline{E}\phi \colon \underline{E}\Lambda \times_{\Lambda} \Gamma \to \underline{E}\Gamma$ be the $\Gamma$-equivariant map induced from the universality of $\underline{E}\Gamma$. 
\begin{lem}\label{lem:Chern}
Assume (\ref{cond:BC3}). Then, there are splittings $s_\Gamma \colon \K^\Gamma_*(\underline{E}\Gamma) \to \K^\Gamma_*(E\Gamma)$ and $s_\Lambda \colon \K^\Gamma_*(\underline{E}\Lambda ) \to \K^\Gamma_*(E\Lambda )$ such that the diagram 
\[\xymatrix{
\K_*^\Lambda(E\Lambda )_\bQ \ar[r]^{E\phi_*} \ar@<1ex>[d]^{(f_\Lambda)_*}  & \K_*^\Gamma(E\Gamma )_\bQ \ar@<1ex>[d]^{(f_\Gamma)_*} \\ 
\K_*^\Lambda(\underline{E}\Lambda)_\bQ \ar[r]^{\underline{E}\phi_*} \ar@<1ex>[u]^{s_\Lambda}  & \K_*^\Gamma(\underline{E}\Gamma )_\bQ \ar@<1ex>[u]^{s_\Gamma}
}\] 
commutes. 
\end{lem}
\begin{proof}
For a discrete group $\Gamma$, we write $F\Gamma$ for the $\bC$-vector space generated by finite order elements of $\Gamma $ on which $\Gamma$ acts by the conjugation. It is proved in \cite[Section 15]{MR928402} (see also \cite[Section 7]{MR1292018}) that the equivariant Chern character 
\[ \ch_\Gamma \colon \K_*^\Gamma (\underline{E}\Gamma ) \to H_*^{\Gamma}(\underline{E}\Gamma , \bC) \cong H_* (\Gamma ; F\Gamma) \]
gives an isomorphism.  By the functoriality of $\ch_\Gamma$, the outer and the inner left and right squares of the diagram
\[ 
\xymatrix{
&&&\\
\K_*(B\Lambda )_\bC \ar[r]_\cong ^{\ch _\Lambda } \ar[d]^{(f_\Lambda)_*} \ar@/^2.5em/[rrr]^{E\phi_*} &H_{*}(\Lambda ; \bC) \ar[r]^{\phi_*} \ar[d]^{(f_\Lambda)_*} & H_{*}(\Gamma ; \bC ) \ar[d]^{(f_\Gamma)_*} & \K_*(B\Gamma)_\bC \ar[l]^\cong _{\ch _\Gamma }  \ar[d]^{(f_\Gamma)_*} \\
\K_*^\Lambda (\underline{E}\Lambda )_\bC \ar[r]_\cong ^{\ch _\Lambda } \ar@/_2.5em/[rrr]^{\underline{E}\phi_*} &H_{*}(\Lambda ; F\Lambda ) \ar[r]^{\phi_*}  & H_{*} (\Gamma ; F\Gamma ) & \K_*^\Gamma( \underline{E}\Gamma )_\bC \ar[l]^\cong _{\ch _\Gamma } \\ &&&
}
\]
commute. Consequently so does the inner middle square.

Now, let $F_0\Gamma $ denote the complement of $\bC \cdot e$ in $F\Gamma$. Then assumption (\ref{cond:BC3}) implies that $\phi$ maps $F_0\Lambda$ to $F_0\Gamma$. Therefore, the projections $F\Lambda \to \bC \cdot e$ and $F \Gamma \to \bC \cdot e$ induce $s_\Lambda \colon H_*(\Lambda; F\Lambda) \to H_*(\Lambda; \bC)$ and $s_\Gamma \colon H_*(\Gamma ; F\Gamma) \to H_*(\Gamma ; \bC)$, which satisfy the desired commutativity.
\end{proof}

\begin{thm}\label{thm:BC}
Let $\phi \colon \Lambda \to \Gamma $ be a homomorphism of groups. Assume (\ref{cond:BC1}), (\ref{cond:BC2}) and (\ref{cond:BC3}). Then the following hold: 
\begin{enumerate}
\item The relative assembly map $\mu_*^{\Gamma, \Lambda}$ is rationally injective.
\item If both $\Lambda $ and $\Gamma$ are torsion-free, then $\mu_*^{\Gamma, \Lambda }$ is split injective.
\item If moreover $\gamma _\Gamma =1$ holds (e.g.\ both $\Gamma $ and $\ker \phi $ have the Haagerup property), then $\mu _*^{\Gamma , \Lambda}$ is an isomorphism.
\end{enumerate}
\end{thm}
\begin{proof}
For (2) and (3), apply the five lemma for the following commutative diagram of exact sequences
\[ 
\xymatrix@C=1em{
\K_*(B\Lambda) \ar[r]^{i_*} \ar[d]^{j_\Lambda(\eta_\Gamma) \circ \mu^\Lambda_* } & \K_*(B\Lambda) \ar[r]^{j_*} \ar[d]^{j_\Gamma(\eta_\Gamma) \circ \mu^\Gamma_* }  & \K_*(B\Gamma, B\Lambda) \ar[r]^{\partial } \ar[d]^{j_\phi(\eta_\Gamma) \circ \mu^{\Gamma, \Lambda}_* }  & \K_{*-1}(B\Lambda ) \ar[r]^{i_*} \ar[d]^{j_\Lambda(\eta_\Gamma) \circ \mu^\Lambda_* } & \K_{*-1}(B\Gamma ) \ar[d]^{j_\Gamma(\eta_\Gamma) \circ \mu^\Gamma_* } \\ 
\K_*(A_\Gamma \rtimes \Lambda ) \ar[r]^{\Phi_{A_\Gamma}} & \K_*(A_\Gamma \rtimes \Gamma ) \ar[r] & \K_*(C\Phi_{A_\Gamma} ) \ar[r] & \K_{*-1}(A_\Gamma \rtimes \Lambda ) \ar[r]^{\Phi_{A_\Gamma}} &\K_*(A_\Gamma \rtimes \Gamma )
}
\]
to see that $j_\phi(\eta_\Gamma) \circ \mu^{\Gamma, \Lambda}_* $ is an isomorphism. Since $j_\phi(\eta_\Gamma)$ has a right inverse $j_\phi(\mathsf{D}_\Gamma )$, we obtain that $\mu^{\Gamma, \Lambda}_*=j_\phi(\mathsf{D}_\Gamma ) \circ j_\phi(\eta_\Gamma) \circ \mu^{\Gamma, \Lambda}_* $ is split injective and isomorphic if $\gamma_\Gamma =1 $.

For (1), we may replace the coefficient field $\bQ$ with $\bC$.  By a diagram chasing similar to the proof of the injectivity part of the five lemma using the sections $s_\Lambda$ and $s_\Gamma$ in Lemma \ref{lem:Chern}, it is checked that the middle map is injective.  
\end{proof}

\begin{thm}\label{thm:nondeg}
Let $\Gamma$ and $\Lambda$ be discrete groups with (\ref{cond:BC1}), (\ref{cond:BC2}) and (\ref{cond:BC3}). Moreover, assume that the pair $(B\Gamma, B\Lambda)$ has the homotopy type of a finite CW-complex. Then the map 
\[ \beta _{\Gamma , \Lambda , \bQ} \colon \K^*(C^*(\Gamma , \Lambda))_\bQ \to \K^*(B\Gamma ,B\Lambda )_\bQ \]
is surjective.
\end{thm}
\begin{proof}
Consider the diagram
\[ \xymatrix{
\K^*(A_\Gamma \rtimes (\Gamma ,\Lambda))_\bQ \ar[r]^{j_\phi (\eta_\Gamma ) } \ar[d] & \K^*(C^*(\Gamma ,\Lambda ) )_\bQ \ar[r]^{\beta_{\Gamma, \Lambda}} \ar[d] & \K^*(B\Gamma , B\Lambda )_\bQ \ar[d] \\
\Hom(\K_*(A_\Gamma \rtimes (\Gamma , \Lambda)),\bQ) \ar[r]^{(j_\phi (\eta_\Gamma ) )^*} & \Hom (\K_*(C^*(\Gamma, \Lambda)), \bQ) \ar[r]^{(\alpha_{\Gamma, \Lambda }) ^*} &  \Hom(\K_*(B\Gamma, B\Lambda), \bQ), 
}\]
where the vertical arrows are the UCT maps. This diagram commutes by the associativity of the Kasparov product and (the proof of) Theorem \ref{thm:BC} implies that the second row is surjective. Moreover, the first and the third vertical morphisms are isomorphic by the universal coefficient theorem for $C_0(B\Gamma \setminus B\Lambda)$ and $A_\Gamma \rtimes (\Gamma, \Lambda)$. This completes the proof of the surjectivity of the first row, and hecne $\beta _{\Gamma, \Lambda}$.
\if0
By the universal coefficient theorem, we obtain an isomorphism
\[ \K^*(B\Gamma, B\Lambda)_\bQ  \cong \Hom (\K_*(B\Gamma , B\Lambda ) , \bQ ).\]
Hence $\beta _{\Gamma , \Lambda , \bQ}$ is surjective if and only if $\Im (\beta _{\Gamma , \Lambda , \bQ})^{\perp } \subset \K_*(B\Gamma, B\Lambda)$ is zero. That is, it suffices to show that for any $x \in \K_*(B\Gamma, B\Lambda)_\bQ$ there is $y \in \K^*(C^*(\Gamma, \Lambda))$ such that 
\[ \alpha_{\Gamma , \Lambda }(x) \hotimes_{C^*(\Gamma ,\Lambda )}  y =x \hotimes _{C_0(B\Gamma^\circ)} \beta_{X,Y}(y) \neq 0 \in \KK(\bC,\bC)_\bQ \cong \bQ.\] 

Let $x \in \K_*(B\Gamma, B\Lambda)_\bQ$. As is shown in the proof of Theorem \ref{thm:BC}, the Kasparov product $\alpha_{\Gamma, \Lambda}(x) \hotimes_{C^*(\Gamma, \Lambda)} j_\phi(\eta_\Gamma)$ is non-zero. 
By assumption (\ref{cond:BC1}) and (\ref{cond:BC2}), both $A_\Gamma \rtimes \Gamma$ and $A_\Gamma \rtimes \Lambda$ are in the UCT class and hence so is $C\Phi_A $. That is, there is $z \in \K^*(C\Phi_A)$ such that 
\[  \alpha_{\Gamma, \Lambda} (x) \hotimes _{C^*(\Gamma , \Lambda )}  j_\phi(\eta_\Gamma ) \hotimes _{C\Phi_{A_\Gamma}} z \neq 0.\]
Now $y := j_{\phi}(\eta_\Gamma) \hotimes _{C\Phi_{A_\Gamma}} z \in \K^*(C^*(\Gamma, \Lambda))$ is the desired element. 
\fi
\end{proof}
\begin{rmk}\label{rmk:nondeg}
We remark that in the proof of Theorem \ref{thm:nondeg} actually shows that, more strongly, the restriction of $\beta_{\Gamma , \Lambda, \bQ}$ on the subgroup $\Im (j_\phi(\eta_\Gamma ) \hotimes {\cdot }) = \Im (j_\phi(\gamma_\Gamma ) \hotimes {\cdot })$ is surjective. 
\end{rmk}
The surjectivity as in Theorem \ref{thm:nondeg} holds even if $(B\Gamma, B\Lambda)$ does not have homotopy type of finite CW-complexes.
\begin{rmk}\label{rmk:nonfinite}
Let us consider the case that $(B\Gamma, B\Lambda)$ does not have the homotopy type of a pair of finite CW-complexes. In this case, the K-group $\K ^*(B\Gamma , B\Lambda )$ is defined as the $\K$-group of the $\sigma$-C*-algebra of continuous function on $B\Gamma$ whose restriction to $B\Lambda$ is zero. The Kasparov product also works in the category of $\sigma$-C*-algebras (for the $\K$-theory and $\KK$-theory of $\sigma$-C*-algebras, see \cite{MR1050490,Bonkat,AranoKubota}), and hence we can define the universal envelope $\beta_{\Gamma, \Lambda }  \colon \K^*(C^*(\Gamma, \Lambda)) \to  \K^*(B\Gamma, B\Lambda)$. 

We have the Milnor $\limone$-sequence \cite[Theorem 3.2]{MR1219733} as
\[ 0 \to {\limone _{(X,Y)}}  \K^{*+1}(X,Y)_\bQ \to \K^* (B\Gamma, B\Lambda)_\bQ \to \varprojlim _{(X,Y)} \K^*(X,Y)_\bQ \to 0, \]
where $(X,Y)$ runs over all finite subcomplexes of $(B\Gamma, B\Lambda)$. Since the groups $\K^*(X,Y)_\bQ$ are finite rank $\bQ$-vector spaces, the projective system $\{ \K^*(X,Y)_\bQ \}_{(X,Y)}$ automatically satisfies the the Mittag-Leffler condition, and hence its $\lim ^1$-term vanishes. That is, we have an isomorphism 
\begin{align*} \K^*(B\Gamma , B\Lambda)_\bQ \cong &  \varprojlim _{(X,Y)} \K^*(X,Y)_\bQ \\  \cong & \Hom (\varinjlim_{(X,Y)} \K_*(X,Y) , \bQ) \\ \cong & \Hom (\K^*(B\Gamma, B\Lambda) , \bQ). \end{align*}
Then the same proof as Theorem \ref{thm:nondeg} works. 
\end{rmk}

For the rest of this section, we study the case that $(\Gamma, \Lambda)$ satisfies (\ref{cond:BC1}), (\ref{cond:BC3}) in page \pageref{cond:BC1} and   
\begin{itemize}
\item[{(\copynum[\label{cond:BC2'}])}] The subgroup $\ker \phi$ of $\Lambda$ is amenable.
\end{itemize}
In this case, $\phi$ induces $\ast$-homomorphisms between reduced group C*-algebras $\phi_r \colon C^*_r\Lambda \to C^*_r\Gamma$ and reduced crossed products $\Phi_{A,r} \colon A \rtimes _r \Lambda \to A \rtimes _r \Gamma$. Let $\epsilon ^{\Gamma, \Lambda } \colon C\phi \to C\phi_r$ denote the quotient. In the same way as the above argument using the universality of $\Kas ^\Gamma$, we get the reduced descent functors $j_{\Gamma, r} $, $j_{\Lambda , r}$ and $j_{\phi , r}$.

\begin{lem}\label{lem:dBCsurj}
Let $\Gamma$ and $\Lambda$ be discrete groups and let $\phi \colon \Lambda \to \Gamma$ be a homomorphism. Assume (\ref{cond:BC1}) and (\ref{cond:BC2'}). Then, the image of the Kasparov product $j_{\phi}(\gamma _\Gamma) $ on the K-homology group $\K^*(C^*(\Gamma, \Lambda))$ is included to the image of $(\epsilon^{\Gamma, \Lambda})^*$.
\end{lem}
\begin{proof}
Since $\Gamma$ acts on $A_\Gamma$ properly, $C\Phi_{A_\Gamma}$ and $C\Phi_{A_\Gamma ,r} $ are isomorphic. (We remark that they are $\KK$-equivalent even if we only assume that $A_\Gamma \in \langle \mathcal{CI} \rangle $). The lemma follows from the commutativity of the diagram
\[\xymatrix@R=1em{
&\K^0(C^*_r(\Gamma, \Lambda ) ) \ar[dd]^{(\epsilon^{\Gamma, \Lambda})^*} \\
\K^0(C\Phi_{A_\Gamma }) \ar[ru]^{j_{\phi ,r }(\eta_\Gamma ) \hotimes {\cdot } } \ar[rd]^{j_{\phi}(\eta_\Gamma ) \hotimes {\cdot }}& \\
&\K^0(C^*(\Gamma, \Lambda))
}  \]
since $\Im (j_\phi (\eta_\Gamma ) \hotimes {\cdot }) = \Im (j_\phi (\gamma _\Gamma) \hotimes {\cdot})$.
\end{proof}

\begin{prp}
Suppose that $(\Gamma, \Lambda)$ satisfies (\ref{cond:BC1}), (\ref{cond:BC2'}) and (\ref{cond:BC3}). Then the composition $\beta_{\Gamma, \Lambda} \circ \epsilon _{\Gamma, \Lambda}$ is rationally surjective.
\end{prp}
\begin{proof}
It follows from Remark \ref{rmk:nondeg} and Lemma \ref{lem:dBCsurj}.
\end{proof}

\appendix
\section{A note on Real Kasparov theory}
In this section, we summarize some calculations of Real KK-theory (in particular the $\KKR_{-1}$-theory) used in this paper. 

\begin{lem}
For a pair $A$, $B$ of $\sigma$-unital Real C*-algebras, the group $\KKR_{-1}(A,B)$ is the set of homotopy classes (in the sense of \cite{MR1656031}) of real self-adjoint Kasparov $A$-$B$ bimodules. 
\end{lem}
Here we say that a triple $(E,\pi, T)$ is a real self-adjoint Kasparov bimodule if $E$ is a countably generated Real Hilbert $B$-module, $\pi \colon A \to \bB(E)$ is a Real $\ast$-representation and $T \in \bB(E)$ is a real self-adjoint operator with $[\varphi (A), T] \in \bK(E)$ and $\varphi(A)(T^2-1) \in \bK(E)$. 
\begin{proof}
Let $\Cl_{0,1}$ be the Clifford algebra generated by a single element $f$ with $f=-f^*$, $\overline{f}=f$ and $f^2=-1$. 
Then the group $\KKR_{-1}(A,B)$ is defined as $\KKR(A \hotimes \Cl_{0,1}, B)$. A Real Kasparov $(A \hotimes \Cl _{0,1})$-$B$ bimodule is a triplet $(E,\pi, F)$, where $E$ is a countably generated $\bZ_2$-graded Hilbert $B$-module, $\varphi \colon A \hotimes \Cl_{0,1} \to \bB(E)$ is a real $\bZ_2$-graded $\ast$-homomorphism and $F$ is a real odd self-adjoint operator with $[\varphi(A \hotimes \Cl_{0,1}), F] \subset \bK(E)$ and $\varphi (A \hotimes \Cl_{0,1})(F^2-1) \subset \bK(E)$. 
Replacing $F$ with its compact perturbation $(F+\varphi(f)F\varphi(f))/2$, we may assume that $F$ anticommutes with $\varphi(f)$. 
Now, the restriction $\varphi(f)|_{E^0}$ gives an isomorphism $E^0 \cong E^1$. Under this identification we have
\[ \varphi(f)= \begin{pmatrix}0 & -1 \\ 1 & 0 \end{pmatrix}_{\textstyle ,} \ \ 
\varphi|_A =\begin{pmatrix} \pi & 0 \\ 0 & \pi \end{pmatrix} \text{ and }
 F= \begin{pmatrix} 0 & T \\ T & 0 \end{pmatrix}_{\textstyle .}
\]
Consequently we get a triple $(E^0,\pi,T)$. This correspondence obviously gives rise to an isomorphism between the set of homotopy classes of real self-adjoint Kasparov $A$-$B$ bimodules and $\KKR_{-1}(A,B)$.
\end{proof}

Next, we give an explicit calculation of the Kasparov product, an odd analogue of \cite[Theorem 18.10.1]{MR1656031}. To this end, we first recall and generalize a useful formula of the Kasparov product. For a $\bZ_2$-graded Hilbert $A$-module $E$, we write $\hat{E}$ for $E \oplus E^{\mathrm{op}}$ and $E^\circ$ for $E$ with the trivial $\bZ_2$-grading.
\begin{lem}\label{lem:Kas}
Let $A$, $B$ and $D$ be $\sigma$-unital Real C*-algebras such that $A$ is separable, let $(E _1, \pi _1, T_1)$ be a real self-adjoint Kasparov $A$-$B$ bimodule and let $(E_2, \varphi_2 , F_2)$ be a real Kasparov $B$-$D$ bimodule. Set $E:=E_1 \otimes _B E_2$, $\pi :=\pi_1 \otimes _B 1$ and $\tilde{T}_1:=T_1 \otimes _B 1$. Let $G=\big( \begin{smallmatrix} 0 & G_0^* \\ G_0 & 0\end{smallmatrix}
\big) \in \bB (E)$ be an odd $F$-connection and assume that $[\pi(A), T] \subset \bK(E)$, where
\[ T= \begin{pmatrix}\tilde{T}_1 & (1-\tilde{T}^2_1)^{1/4}G^*_0 (1-\tilde{T}^2_1)^{1/4} \\  (1-\tilde{T}^2_1)^{1/4}G_0 (1-\tilde{T}^2_1)^{1/4} & -\tilde{T}_1 \end{pmatrix} \in \bB(E). \]
Then, the real self-adjoint Kasparov $A$-$D$ bimodule $(E^\circ , \pi, T)$ represents the Kasparov product $[E_1,\pi_1, T_1] \hotimes_B [E_2, \pi_2 , F_2]$.
\end{lem}
\begin{proof}
Let $(\hat{E}_1,\varphi_1, F_1)$ be the real Kasparov $(A \hotimes \Cl_{0,1})$-$B$ bimodule corresponding to $(E_1, \varphi_1, F_1)$, that is, $\varphi _1(f)=\big( \begin{smallmatrix} 0 & -1\\ 1 & 0 \end{smallmatrix} \big)
$, $\varphi _1 |_A = \big( \begin{smallmatrix} \pi_1 & 0\\ 0 & \pi_1 \end{smallmatrix} \big)$ and $F_1= \big( \begin{smallmatrix} 0 & T_1\\ T_1 & 0 \end{smallmatrix}\big) $. 
Then, the $\bZ_2$-graded unitary $U \colon \hat{E}_1 \hotimes E_2 \cong \widehat{E^\circ}$ given by 
\[ (e_1^0 \oplus e_1^1) \hotimes (e_2^0 \oplus e_2^1)  
\mapsto  ((e_1^0 \otimes e_2^0) \oplus (e_1^1 \otimes e_2^1)) \oplus ((e_1^1 \otimes e_2^0) \oplus (-e_1^0 \otimes e_2^1)) 
\]
satisfies $U(\varphi_1 \hotimes 1)(f)U^*=\big( \begin{smallmatrix}0 & -1 \\ 1 & 0 \end{smallmatrix} \big)$.

Set $K_0:=(1-\tilde{T}^2_1)^{1/4}G_0 (1-\tilde{T}^2_1)^{1/4}$ and  
\begin{align*}
 F&= U(F_1 \hotimes 1 + ((1-F_1^2)^{1/4}\hotimes 1)G((1-F_1^2)^{1/4} \hotimes 1))U^*\\
&=\begin{pmatrix}
0 & 0 & \tilde{T}_1 & K_0^* \\
0 & 0 & K_0 & -\tilde{T}_1 \\
\tilde{T}_1 & K_0^* & 0 & 0 \\ 
K_0 & -\tilde{T}_1 & 0 & 0
\end{pmatrix}
=\begin{pmatrix}
0 & T \\ T & 0
\end{pmatrix}
\in \bB(\widehat{E^\circ}).
\end{align*}
Now we apply \cite[Theorem 18.10.1]{MR1656031}, which also holds for Real $\KK$-groups, to conclude that the Kasparov bimodule $(\widehat{E^\circ}, \mathop{\mathrm{Ad}} U \circ \varphi  , F)$, which corresponds to the self-adjoint Kasparov $A$-$D$ bimodule $(E^\circ,\pi,T)$, represents the Kasparov product $[E_1, \pi _1,T_2] \hotimes _B [E,\pi_2,F_2]$. 
\end{proof}

\begin{rmk}\label{rmk:KR-1}
In a similar fashion to the case of complex $\K_1$-theory, there is a unitary description of the $\KR_{-1}$-group (a reference is \cite{MR3576286}). For a Real C*-algebra $A$, the group $\KR_{-1}(A)$ is defined as the set of homotopy classes of transpose-invariant unitaries, that is, 
\begin{align} \KR_{-1}(A):= \{ u \in 1+A \otimes \bK \mid uu^*=u^*u=1,  u^*=\overline{u} \}/\sim _{\text{homotopy}}. \label{form:KR-1} \end{align}
The identification
\[ \KKR_{-1} (\bR, A)  \cong \KR_{-1}(A) \]
is given by the composition of the isomorphism $\KKR_{-1}(\bR, A) \cong \KR_0(Q_s(A))$ and the boundary homomorphism. It is described explicitly as
\begin{align*}
 [E, 1, T] \mapsto [-\exp (-\pi i T)] \in \KR_{-1}(\bK(E)) \cong \KR_{-1}(B).
\end{align*}
It $T$ is written as $D(1+D^2)^{-1/2}$ by a Real self-adjoint unbounded Kasparov bimodule $[E,1, D]$, then $-\exp (-\pi i T)$ is homotopic to the Cayley transform $\Cay (D)$ of $D$.
\end{rmk}

\bibliographystyle{alpha}
\bibliography{bibABC,bibDEFG,bibHIJK,bibLMN,bibOPQR,bibSTUV,bibWXYZ,arXiv}

\end{document}